\documentclass[12pt]{article}

\usepackage[utf8]{inputenc}
\usepackage[T1]{fontenc}
\usepackage{mathptmx}
\usepackage{amsmath,amssymb,amsthm}
\usepackage{mathtools}
\usepackage{mathrsfs}
\usepackage{graphicx}
\usepackage[letterpaper,margin=1in]{geometry}
\usepackage{array}
\usepackage{enumerate}
\renewenvironment{abstract}
	{\quotation}
	{\endquotation}

\date{}

\makeatletter
\renewcommand{\fnum@figure}{\textbf{Fig. \thefigure}}
\renewcommand{\fnum@table}{\textbf{Table \thetable}}
\long\def\@makecaption#1#2{%
  \vskip\abovecaptionskip
  \sbox\@tempboxa{#1. #2}%
  \ifdim \wd\@tempboxa >\hsize
    #1. #2\par
  \else
    \global\@minipagefalse
    \hb@xt@\hsize{\hfil\box\@tempboxa\hfil}%
  \fi
  \vskip\belowcaptionskip}
\makeatother

\usepackage{scicite}
\usepackage{url}
\usepackage[colorlinks=true,linkcolor=blue,citecolor=blue,urlcolor=blue]{hyperref}

\newcommand{\aLeb}{a_{\mathrm{Leb}}}
\newcommand{\area}{\operatorname{area}}
\newcommand{\R}{\mathbb{R}}
\newcommand{\Z}{\mathbb{Z}}

\def\scititle{
	An exact hierarchy for Lebesgue's universal covering constant and a certified 0.834 lower bound
}

\title{\bfseries \boldmath \scititle}

\author{
	Shuai~Zeng$^{1\ast}$\\
	\small$^{1}$Chongqing University of Posts and Telecommunications, Chongqing, China.\\
	\small$^\ast$Corresponding author. Email: zengshuai@cqupt.edu.cn
}

\newtheorem{theorem}{Theorem}[section]
\newtheorem{lemma}[theorem]{Lemma}
\newtheorem{proposition}[theorem]{Proposition}
\newtheorem{corollary}[theorem]{Corollary}

\theoremstyle{definition}

\theoremstyle{remark}

\begin{document}
\maketitle

\begin{abstract} \bfseries \boldmath
Posed by Lebesgue in 1914, the universal covering problem asks for the
smallest-area planar convex set containing a congruent copy of every set
of diameter at most one.  We introduce an exact Reuleaux-type
variational hierarchy for this constant: its monotone finite-arc values
\(\Lambda_M\) satisfy \(\aLeb=\lim_{M\to\infty}\Lambda_M\), and each
level is a continuous finite-dimensional problem.  We prove
\(0\le \aLeb-\Lambda_M\le C M^{-2}\), giving a controlled finite-arc
route to the constant itself.  As a certified low-order realization, an
outward-rounded interval certificate for a regular finite Reuleaux
subtest proves \(\aLeb\ge0.834\), improving the lower-bound benchmark
established by Brass and Sharifi in 2005.
\end{abstract}

\section*{Introduction}

Lebesgue's universal covering problem asks for the least possible area
of a compact planar convex set that contains a congruent copy of every
nonempty compact planar set of diameter at most one.  More precisely, a
compact convex set \(U\) is called a universal cover if, for every
nonempty compact planar set \(K\) with diameter at most one, some
Euclidean isometry \(g\) satisfies \(g(K)\subset U\).  The corresponding
Lebesgue universal covering constant, denoted \(\aLeb\), is the infimum
of the areas of all compact convex universal covers.

The problem was posed by Lebesgue in 1914 in a letter to P\'al and has
remained open for more than a century~\cite{Pal1920}.  P\'al showed that
a regular hexagon with inradius \(1/2\) is a universal cover and then
improved this construction by removing two corners, obtaining the
classical upper bound
\[
\aLeb\le 2-\frac{2}{\sqrt 3}\approx0.8453.
\]
Subsequent work by Sprague, Hansen, Baez--Bagdasaryan--Gibbs, and Gibbs
further reduced the known explicit upper bounds by increasingly delicate
geometric cuts from previously known covers~\cite{Sprague1936,
Hansen1992,BaezBagdasaryanGibbs2015,Gibbs2018}.

The lower-bound side has a different character.  An upper bound is
proved by exhibiting one universal cover of a given area.  A lower bound
must rule out every convex universal cover below a proposed area, and
therefore must identify geometric constraints that any universal cover
must satisfy.  Brass and Sharifi improved the lower-bound benchmark to
\(\aLeb\ge0.832\) in 2005 by using a finite collection of extremal test
shapes and analyzing the area of their convex hull under optimal relative
placements~\cite{BrassSharifi2005}.  That argument gives a rigorous
finite-test obstruction.

The known upper and lower constructions are rigorous and geometric, but
they do not by themselves organize the full covering constant into a
convergent finite-dimensional variational framework.  This paper supplies
such a framework for the century-old problem: it organizes the infinite
family of test sets into finite levels whose values converge to the
original covering constant.  The hierarchy constructed here embeds finite
lower-bound tests in a monotone sequence of continuous finite-dimensional
finite-arc problems, where finite-arc means that the boundary has a
bounded number of circular arcs.

The first main result is exactness:
\[
\aLeb=\lim_{M\to\infty}\Lambda_M.
\]
Here \(\Lambda_M\) denotes the \(M\)-th finite-arc variational value
defined below.
The Lebesgue constant is recovered as the limit of an explicit
Reuleaux-type hierarchy.  The second main result is quantitative: for an
absolute constant \(C\),
\[
0\le \aLeb-\Lambda_M\le C M^{-2}.
\]
This gives an a priori scale for the loss incurred by truncating the
hierarchy at finite arc complexity \(M\).  The third result is a certified
low-order realization.  From the second hierarchy level we extract a
regular finite Reuleaux subtest, reduce its relative placement problem to
a compact five-dimensional domain, and prove by outward-rounded interval
arithmetic that this finite test already forces area at least \(0.834\).
Hence
\[
\aLeb\ge0.834.
\]
This improves the lower-bound benchmark that had stood since 2005 and
demonstrates that the hierarchy can be turned into rigorous finite
certificates.  The bound is a low-order realization of the general
hierarchy; the exactness and quantitative convergence theorems provide
the route for systematic finite-level improvement.

The logical relation among these results is shown in Fig.~\ref{fig:hierarchy}.
The main text states the variational framework, the exactness theorem,
the quantitative convergence theorem, and the certified low-order
application.  The Supplementary Materials give the full mathematical
proofs and certificate-level details.

\begin{figure}[!t]
\centering
\includegraphics[width=0.84\textwidth]{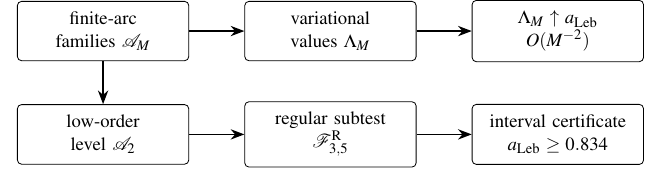}
\caption{\textbf{From infinite covering constraints to an exact finite-arc hierarchy.}
The top row records the Reuleaux-type hierarchy, whose values
\(\Lambda_M\) increase to \(\aLeb\) with \(O(M^{-2})\) truncation error.
The bottom row shows the certified low-order realization used here: a
regular finite Reuleaux subtest is extracted from the second level and
verified by interval arithmetic.}
\label{fig:hierarchy}
\end{figure}

\section*{Results}

\subsection*{Reduction to constant width and finite-arc hierarchy}

The first reduction is classical.  It is enough to test bodies of
constant width one.  More precisely, let \(K\) be a nonempty compact
convex planar test set of diameter at most one.  Then \(K\) is contained
in a planar body of constant width one; conversely, every width-one
constant-width body has diameter one.  Hence a compact convex set is a
universal cover exactly when it contains a congruent copy of every planar
body of constant width one~\cite{BonnesenFenchel1934,Eggleston1958,
ChakerianGroemer1983,Schneider2014}.  We denote this class by
\(\mathcal W_1\).

This reduction is important because it separates the covering problem
from the arbitrary topology of diameter-one sets.  The original
quantifier ranges over all nonempty compact planar sets of diameter at
most one.  After convexification and completion, the same constant is
obtained by testing the geometrically structured class
\(\mathcal W_1\).  This class is still large: a general body of constant
width can have a nonsmooth boundary and a continuous curvature
distribution.  The hierarchy below is designed to approximate this
infinite-dimensional class without leaving the category of legitimate
width-one test bodies.  The full problem formulation and the
constant-width reduction are given in Supplementary Materials, sections
S1 and S2.

The class \(\mathcal W_1\) is still infinite-dimensional.  The hierarchy
is built by replacing it with increasing finite-arc Reuleaux-type
subfamilies.  Let \(D\) be the disk of radius \(1/2\).  For each odd
\(n\ge3\), let \(\mathcal R_n\) denote the family of width-one
Reuleaux-type constant-width bodies whose boundary consists of \(n\)
circular-arc components.  The \(M\)-th hierarchy level is
\[
\mathcal A_M
=
\{D\}\cup\bigcup_{m=1}^M\mathcal R_{2m+1}.
\]
The disk is included as the infinite-arc endpoint of the construction.

Each \(\mathcal R_n\) is a continuous finite-dimensional geometric
family.  After quotienting translations, its members are described by
finitely many arc-transition angles subject to closing, antipodal,
ordering, and nondegeneracy constraints.  A fixed hierarchy level has
bounded finite-arc complexity while still containing continuous families
of shapes.
Concretely, a nondegenerate member of \(\mathcal R_n\) is encoded by the
\(2n\) transition angles of the \(0/1\) curvature-radius density on
\(S^1\), the unit circle.  The antipodal relation pairs transition data,
the two closing equations impose the vanishing first moment, cyclic
ordering gives inequalities, and rotations act by adding a common angle.
After the usual Euclidean quotient, each level is described by finitely
many real variables and finitely many equations and inequalities.
The finite-arc families and the variational hierarchy are defined in
detail in Supplementary Materials, sections S3 and S4.

For a test family \(\mathcal F\), define \(\Lambda(\mathcal F)\) to be
the least area of a convex set containing a congruent copy of every
member of \(\mathcal F\).  Set
\[
\Lambda_M=\Lambda(\mathcal A_M).
\]
Since \(\mathcal A_M\subset \mathcal A_{M+1}\subset\mathcal W_1\), the
values satisfy
\[
\Lambda_M\le\Lambda_{M+1}\le\aLeb.
\]
This gives the lower-bound direction of the hierarchy: finite levels give
valid lower bounds, and adding more finite-arc complexity can only
increase the value.

At fixed \(M\), the class \(\mathcal A_M\) is a finite union of
continuous finite-dimensional families of constant-width bodies, with a
uniform bound on the number of circular arcs.  The associated value
\(\Lambda_M\) is therefore a finite-dimensional geometric optimization
problem over bodies of bounded finite-arc complexity.  These stages
organize the original infinite-dimensional search while retaining a
limiting connection to the original constant.

For a finite subfamily \(\mathcal F=\{K_1,\ldots,K_N\}\), the same
functional has the concrete convex-hull form
\[
\Lambda(\mathcal F)
=
\inf_{g_1,\ldots,g_N}
\area\,
\operatorname{conv}\left(\bigcup_{j=1}^N g_jK_j\right),
\]
where the infimum is over Euclidean placements.  This formula is the
bridge from the abstract hierarchy to finite-dimensional certified
placement problems: any universal cover containing one congruent copy of
each \(K_j\) must also contain the convex hull of such a common
placement.  The finite-family convex-hull formula is proved in
Supplementary Materials, section S4.

The certified computation below uses one finite subfamily extracted from
this hierarchy.  The exactness theorem places this low-order computation
in a framework whose finite-arc levels converge to \(\aLeb\).

\subsection*{Exactness of the hierarchy}

The main structural theorem is
\[
\aLeb=\lim_{M\to\infty}\Lambda_M.
\]
Passing from all constant-width bodies to finite-arc Reuleaux-type
families loses no information in the limit.  The theorem gives a limiting
variational characterization of \(\aLeb\) itself.

Equivalently, for every \(\varepsilon>0\) there is a finite-arc level
\(M\) such that the optimal lower-bound value obtained from
\(\mathcal A_M\) is within \(\varepsilon\) of the true Lebesgue constant.
This statement addresses the structural issue left open by isolated
finite tests: it gives a convergent sequence of test families whose
limiting value is the original universal covering constant.

The proof uses the curvature-radius representation of bodies of constant
width.  A width-one body can be described by a curvature-radius density
\(\rho(\theta)\) satisfying \(0\le\rho\le1\), together with linear
closing and antipodal constraints.  Finite-arc Reuleaux-type bodies are
the bang-bang cases in which \(\rho\) takes only the values \(0\) and
\(1\), with finitely many switches.  Intervals where \(\rho=1\) give
circular arcs, while intervals where \(\rho=0\) give vertex normal cones.

The key approximation step is a finite-switch bang-bang purification
principle.  On a prescribed angular mesh, the curvature-radius density is
replaced by a \(0/1\)-valued step density that preserves the relevant
moments.  These moments encode the closing and antipodal conditions
needed for a legitimate constant-width body.  The number of switches is
controlled by the number of moment constraints, so the purified body has
finite arc complexity.  The result is a legal finite-arc Reuleaux-type
body close to the original constant-width body.  The underlying
purification step follows from a one-dimensional finite-switch selection
principle based on nonatomic vector-measure purification and affine
Hobby--Rice arguments
~\cite{Lyapunov1940,DvoretzkyWaldWolfowitz1951,HobbyRice1965,Warga1972}.

Compactness then completes the exactness argument.  If the hierarchy did
not converge to \(\aLeb\), one could take a sequence of nearly optimal
covers for all finite levels and pass to a Hausdorff limit.  Density of
the finite-arc Reuleaux-type bodies would force the limiting convex set
to contain every constant-width body, hence to be a universal cover.
This contradicts the assumed gap below \(\aLeb\).  The full proof,
including the finite-switch statement and the density argument, is given
in Supplementary Materials, section S5.

The disk \(D\) plays a technical and geometric role in this construction.
It is included explicitly as the round constant-width body and may be
viewed as the infinite-arc endpoint of the finite-arc families.  Including
it prevents the hierarchy from treating the round case as a limiting
exception and keeps the finite levels compatible with the compactness
arguments used in the proof.

\subsection*{Quantitative convergence}

Exactness is quantitative.  There exists a constant \(C\) such that
\[
0\le \aLeb-\Lambda_M\le C M^{-2}.
\]
After the hierarchy is truncated at level \(M\), the remaining possible
loss in the variational value is at most order \(M^{-2}\).  The same
estimate gives tail stability:
\[
0\le \Lambda_N-\Lambda_M\le C M^{-2}\qquad (N\ge M).
\]
Later hierarchy levels can change the lower-bound value by at most order
\(M^{-2}\) after the \(M\)-th level has been reached.

The proof has three steps.  First, local moment matching produces
finite-switch bang-bang approximants while preserving the linear
constraints that define width-one bodies.  Second, a Green-kernel
estimate converts the moment control into a second-order support-function
and Hausdorff approximation bound for finite-arc Reuleaux-type bodies.
Third, a one-sided area-transfer argument converts Hausdorff
approximation into a covering-area estimate.  If a convex set covers the
finite-arc approximants, then a parallel enlargement of controlled width
covers the corresponding constant-width bodies; the resulting area
increase is controlled by a Steiner-type estimate.  This gives the
stated \(M^{-2}\) upper bound for \(\aLeb-\Lambda_M\).  The Green-kernel
estimate, the Hausdorff approximation bound, and the area-transfer lemma
are proved in Supplementary Materials, section S6.

Moment constraints produce the second-order estimate.  A naive finite-arc
approximation would only control first-order boundary errors; the moment
conditions cancel the leading local error in the support function, and
the Green kernel for the one-dimensional support equation then gives a
quadratic mesh-size bound.  Since the number of available arcs grows
linearly with \(M\), the mesh scale is \(M^{-1}\) and the resulting
Hausdorff approximation scale is \(M^{-2}\).  The area-transfer step
preserves this order at the level of the covering functional.

The estimate controls the truncation error of the full finite-arc
hierarchy and gives the tail bound used to guide finite certified
realizations.

\subsection*{Certified low-order realization and the 0.834 bound}

The certified numerical result is a low-order realization of the
hierarchy.  The second hierarchy level contains the disk and the first
two nontrivial odd Reuleaux-type families:
\[
\mathcal A_2=\{D\}\cup\mathcal R_3\cup\mathcal R_5.
\]
From this level we extract the regular finite Reuleaux subtest
\[
\mathcal F_{3,5}^{\rm R}=\{D,R_3,R_5\},
\]
where \(R_3\) and \(R_5\) are the regular Reuleaux triangle and regular
Reuleaux pentagon of width one.  Its finite-test value
\[
L_{3,5}^{\rm R}:=\Lambda(\mathcal F_{3,5}^{\rm R})
\]
is automatically a lower bound for \(\aLeb\) once it is certified from
below, because \(\mathcal F_{3,5}^{\rm R}\subset \mathcal W_1\).  The
finite-test value and its convex-hull formulation are detailed in
Supplementary Materials, sections S7 and S8.

After fixing the disk and using the symmetries of the regular Reuleaux
triangle and pentagon, the common-placement problem reduces to five
parameters.  If
\[
z=(x_3,y_3,\psi_5,x_5,y_5)
\]
denotes the normalized placement vector, the area functional is
\[
A(z)=
\area\,
\operatorname{conv}\bigl(D\cup R_3(z)\cup R_5(z)\bigr).
\]
With the complete normalization,
\[
L_{3,5}^{\rm R}=\inf_z A(z),
\]
where \(z\) ranges over all admissible normalized configurations.  The
three test bodies and the associated five-parameter model are shown in
Fig.~\ref{fig:finite-test}.  The complete normalization is proved in
Supplementary Materials, section S9.

\begin{figure}[!t]
\centering
\includegraphics[width=0.98\textwidth]{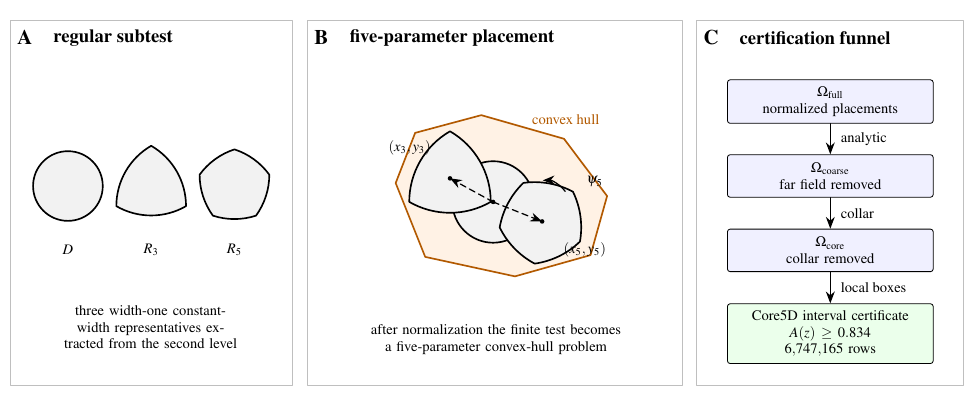}
\caption{\textbf{The certified finite Reuleaux realization.}
(A) The regular finite subtest consists of the disk \(D\), the regular
Reuleaux triangle \(R_3\), and the regular Reuleaux pentagon \(R_5\), all
of width one.  (B) After normalization, a common placement is described
by five parameters and gives a convex-hull area function \(A(z)\).
(C) The interval certificate proves \(A(z)\ge0.834\) by reducing the
full parameter domain to a compact core and certifying the remaining
boxes.}
\label{fig:finite-test}
\end{figure}

Analytic far-field estimates and collar certificates reduce the
unbounded normalized parameter domain to the compact core box
\[
\Omega_{\rm core}
=
[-0.08,0.08]^2\times[0,\pi/5]\times[-0.08,0.08]^2.
\]
On this core domain, a finite outward-rounded interval certificate proves
that \(A(z)\ge0.834\) for every \(z\).  The certificate records a
subdivision and local geometric templates, and the verifier checks the
domain reductions, box coverage, and interval lower bounds using
outward-rounded interval
arithmetic~\cite{Moore1966,Neumaier1990,MooreKearfottCloud2009,
Tucker2011}.  The main components are summarized in
Table~\ref{tab:certificate}.  The analytic reductions, core-domain
certificate, and global certification theorem are described in
Supplementary Materials, sections S10 to S14.

\begin{table}[!t]
\caption{\textbf{Main components of the interval certificate.}}
\centering
\begin{tabular}{p{0.34\textwidth}p{0.56\textwidth}}
\hline
component & role \\
\hline
far-field estimates & remove configurations with large translations \\
collar certificates & reduce the coarse bounded domain to the core box \\
Core5D certificate & cover the remaining five-dimensional core domain by interval boxes \\
local methods & certify convex-hull lower bounds on each leaf box \\
deterministic verifier & recompute all certified inequalities from archived data \\
\hline
\end{tabular}
\label{tab:certificate}
\end{table}

The certificate has three local mechanisms.  The direct inner-polygon
method constructs a certified polygonal region inside the true convex
hull and bounds its shoelace area from below.  The template convexity
method uses a fixed ordered template whose convexity and area remain
certified throughout a parameter box.  The simple-template mean-value
method proves a lower area bound by evaluating the center area and a
rigorous derivative enclosure over the box.  In each case the certified
polygonal region is contained in the true convex hull, so a certified
area lower bound for the polygon is a certified lower bound for \(A(z)\).
The local certificate methods are specified in Supplementary Materials,
section S13.

The archived certificate contains \(6,747,165\) records.  The final audit
has certification status PASS, zero errors and warnings, and a minimum
strict certificate margin above \(0.834\) larger than \(2.6\cdot10^{-11}\).
Combining the far-field reductions, the collar certificates, and the
Core5D interval certificate gives
\[
L_{3,5}^{\rm R}\ge0.834.
\]
Since \(L_{3,5}^{\rm R}\le\aLeb\), this proves
\[
\aLeb\ge0.834.
\]
This improves the best previously published lower-bound benchmark for
the problem, due to Brass and Sharifi in 2005.  The certificate data and
verifier are archived with a persistent identifier~\cite{CertificateArchive};
an interactive visualization of the five-parameter placement model is
also available~\cite{PlacementDemo}.  The audit records and final
soundness theorem are given in Supplementary Materials, section S14.

The finite Reuleaux subtest also clarifies the relation to the previous
polygonal finite-test method.  The regular Reuleaux triangle and regular
Reuleaux pentagon contain the corresponding regular polygonal test bodies
used in that comparison and complete them to bodies of constant width.
The same placements that produce a polygonal obstruction are subsumed by
constant-width test bodies, which are legitimate members of the full
reduced problem.  Consequently, at the level of finite-test
values, the Reuleaux test is no weaker than the corresponding polygonal
finite test.  The interval certificate turns this geometric strengthening
into the certified bound \(0.834\).

\section*{Discussion}

The hierarchy organizes the constant-width search into finite-dimensional
finite-arc variational problems and recovers the covering constant in the
limit:
\[
\Lambda_M\uparrow\aLeb.
\]
The quantitative theorem gives a controlled tail estimate:
\[
0\le \aLeb-\Lambda_M\le C M^{-2}.
\]
The certified \(0.834\) bound is the first low-order realization of this
framework: a finite Reuleaux subtest extracted from the hierarchy is
reduced to a finite-dimensional placement problem and certified by
outward-rounded interval arithmetic.  Higher levels, richer finite
subtests, and stronger local interval templates provide a path toward
sharper certified lower bounds approaching \(\aLeb\).

Several mathematical questions remain open.  Although the Reuleaux
finite-arc approximation error has the \(M^{-2}\) scale, the
covering-area functional may smooth some Hausdorff-level approximation
error.  Determining the true asymptotic behavior of
\(\aLeb-\Lambda_M\) remains an interesting problem.  The near-minimizers
of the finite hierarchy levels also deserve further study.  Their
geometry may reveal features of extremal or near-extremal universal
covers, identify informative finite tests, and guide improved upper-bound
constructions.  More broadly, the approach shows how an infinite
geometric covering problem can be organized into finite-dimensional
variational tests together with independently checkable interval
certificates.  This separation between variational search and rigorous
verification may be useful in other extremal problems whose obstruction
sets are infinite but admit structured finite approximants.

\section*{Materials and Methods}

\subsection*{Constant-width reduction}

The original universal-covering problem involves all nonempty compact
sets of diameter at most one.  The first reduction replaces this class by
the planar bodies of constant width one.  Convexification does not
increase diameter, and every compact convex planar set of diameter at
most one is contained in a complete convex body of the same diameter.  In
the plane, complete convex bodies are precisely bodies of constant width
equal to their diameter.  If the diameter is smaller than one, Minkowski
addition with a disk increases the width to one.  This gives the test
class \(\mathcal W_1\) used throughout the paper.  Details are provided
in Supplementary Materials, sections S1 and S2.

\subsection*{Curvature-radius representation and finite arcs}

A width-one constant-width body can be represented by its support
function or, equivalently, by curvature-radius data along the normal
angle.  In this representation, the width-one condition becomes a linear
antipodal constraint, and the closing condition becomes a pair of first
harmonic moment constraints.  General members of \(\mathcal W_1\) may
have nonsmooth boundary and arbitrary curvature-radius distributions.
Finite-arc Reuleaux-type bodies correspond to the special case in which
the curvature-radius density is a step function taking only the values
\(0\) and \(1\) and having finitely many switches.

For fixed odd \(n\), the switch angles give a finite set of parameters.
Closing, antipodal, ordering, and nondegeneracy conditions cut out a
finite-dimensional admissible parameter domain.  Each hierarchy level is
finite-dimensional in this geometric variational sense: it is a
continuous family of bounded finite-arc complexity.
Equivalently, one may use the \(2n\) transition angles of the step
density as local coordinates on each cyclic-ordering cell.  The
antipodal relation, the two closing equations, and the nondegeneracy
inequalities are finite constraints in these coordinates.  No infinite
functional parameter remains at a fixed arc level.  The construction of
\(\mathcal R_n\), \(\mathcal A_M\), and \(\Lambda_M\) is given in
Supplementary Materials, sections S3 and S4.

\subsection*{Bang-bang purification and exactness}

The exactness proof uses a one-dimensional bang-bang purification
principle.  On each angular mesh, one imposes finitely many moment
constraints encoding the first harmonics and the local approximation
moments.  A nonatomic vector-measure purification theorem, equivalently a
finite-switch affine Hobby--Rice principle in this setting, replaces an
admissible density \(0\le\rho\le1\) by a \(0/1\)-valued density with
controlled switch count while preserving the imposed moments.  The
resulting bang-bang density defines a legal finite-arc constant-width
body.  Refining the mesh gives finite-arc bodies converging to the
original constant-width body, which yields density of the hierarchy and
hence exactness by compactness.  The finite-switch purification theorem
and the compactness proof are in Supplementary Materials, section S5.

\subsection*{Quantitative area transfer}

The \(M^{-2}\) estimate comes from a second-order approximation scale.
The moment matching on each mesh interval cancels the leading local
support-function error.  A Green-kernel estimate then bounds the
Hausdorff error between an arbitrary constant-width body and a
finite-arc Reuleaux-type approximant by order \(M^{-2}\).  If a convex
set covers all approximants at level \(M\), then a parallel enlargement
by this Hausdorff radius covers all width-one constant-width bodies.  The
Steiner formula bounds the corresponding area increase, giving
\[
0\le \aLeb-\Lambda_M\le C M^{-2}.
\]
The full quantitative proof, including the Green-kernel estimate and the
area-transfer lemma, is in Supplementary Materials, section S6.

\subsection*{Five-parameter normalization}

For the certified finite subtest, the disk is fixed at the origin.
Global Euclidean freedom and the dihedral symmetries of the regular
Reuleaux triangle fix the orientation of \(R_3\).  The remaining degrees
of freedom are the translation of \(R_3\), the orientation of \(R_5\) in a
fundamental angular interval, and the translation of \(R_5\).  This gives
the five-parameter vector
\[
z=(x_3,y_3,\psi_5,x_5,y_5).
\]
The complete normalization proves that minimizing over this domain is
equivalent to minimizing over all relative congruent placements of the
three test bodies.  The normalization proof is in Supplementary
Materials, section S9.

\subsection*{Interval certification workflow}

The interval certificate separates search from verification.  The
certificate records finite domain data, witness information, and local
method labels; the verifier recomputes all interval lower bounds using
outward-rounded interval arithmetic.  The proof obligation is to show
\[
A(z)\ge0.834
\]
for all normalized configurations \(z\).  Analytic far-field estimates
exclude large translations.  One-body collar certificates reduce the
remaining coarse domain to the compact core box.  The core box is covered
by interval boxes, and each accepted leaf box is certified by one of the
approved local geometric methods.  The global theorem follows because
the certified regions cover the full normalized parameter domain and
each region has a rigorous interval lower bound at target \(0.834\).  The
far-field, collar, core-domain, and global soundness arguments are given
in Supplementary Materials, sections S10 to S12 and S14.

\subsection*{Local lower-bound certificates}

The local certificates all prove the same geometric statement: throughout
an interval box \(B\) in the five-parameter domain, the convex hull
\(\operatorname{conv}(D\cup R_3(z)\cup R_5(z))\) has area at least
\(0.834\).  The direct inner-polygon method selects body-labeled witness
points, verifies their membership by interval arithmetic, and applies an
interval shoelace area bound.  The template convexity method constructs
an ordered polygonal template whose perturbation over \(B\) remains
convex and whose area stays above the target.  The simple-template
mean-value method certifies a simple positively oriented polygon at the
box center and subtracts a rigorous derivative bound over \(B\).  These
methods are mathematically interchangeable at the global level: a leaf
box is accepted only when at least one approved method gives a rigorous
interval lower bound.  The interval predicates for membership,
convexity, simplicity, orientation, and area are specified in
Supplementary Materials, section S13.

\subsection*{Certificate audit}

The archive separates certificate generation from certificate
verification.  For each accepted box, the verifier recomputes the
required geometric inclusions, convexity or simplicity conditions, and
interval area lower bound from the recorded finite data.  The final audit
checks the displayed constants, analytic reductions, collar regions, box
coverage, and local lower-bound certificates.  The final status table
reports PASS, zero errors and warnings, and strict interval margins above
the target value.  The Supplementary Materials give the detailed
certificate schema and the complete global soundness theorem
(Supplementary Materials, sections S12 and S14).

\section*{Acknowledgments}

\paragraph*{Author contributions:}
S.Z. conceived the project, developed the mathematical framework,
implemented the certification workflow, analyzed the results, and wrote
the manuscript.

\paragraph*{Competing interests:}
The author declares no competing interests.

\paragraph*{AI use:}
Large language models were used for language editing, formatting
assistance, and consistency checks.  The author reviewed all such outputs
and is fully responsible for the manuscript.

\paragraph*{Data, code, and materials availability:}
The archived repository contains the certificate data, standalone
verifier, and audit report~\cite{CertificateArchive}.  A visualization
demo is available online~\cite{PlacementDemo}.  No physical materials
were generated in this work.

\clearpage
\linespread{1}\selectfont
\setlength{\emergencystretch}{2em}
\renewcommand{\thesection}{S\arabic{section}}
\renewcommand{\thefigure}{S\arabic{figure}}
\renewcommand{\thetable}{S\arabic{table}}
\renewcommand{\theequation}{S\arabic{equation}}
\setcounter{section}{0}
\setcounter{figure}{0}
\setcounter{table}{0}
\setcounter{equation}{0}

\begin{center}
{\Large\bfseries Supplementary Materials for}\\[0.6em]
{\large\bfseries An exact hierarchy for Lebesgue's universal covering constant and a certified 0.834 lower bound}\\[0.8em]
Shuai~Zeng$^{1\ast}$\\
{\small $^{1}$Chongqing University of Posts and Telecommunications, Chongqing, China.}\\
{\small $^\ast$Corresponding author. Email: zengshuai@cqupt.edu.cn}
\end{center}

\begin{center}
\textbf{This PDF file includes:}\\
Materials and Methods\\
References and Notes
\end{center}

\nocite{Pal1920,Sprague1936,Hansen1992,BaezBagdasaryanGibbs2015,Gibbs2018,
BrassSharifi2005,BonnesenFenchel1934,Eggleston1958,ChakerianGroemer1983,
Schneider2014,Lyapunov1940,DvoretzkyWaldWolfowitz1951,HobbyRice1965,
Warga1972,Moore1966,Neumaier1990,MooreKearfottCloud2009,Tucker2011,
CertificateArchive,PlacementDemo}

\section*{Materials and Methods}

The main text states the logical structure, theorem statements, and proof
mechanisms.  These supplementary materials provide the full mathematical
proofs and certificate-level details.

The sections below give the mathematical and computational methods used
in the main manuscript.  Sections~S1--S6 contain the constant-width
reduction, the Reuleaux-type variational hierarchy, exactness, and
quantitative convergence.  Sections~S7--S14 contain the finite Reuleaux
subtest, the five-parameter normalization, analytic domain reductions,
the Core5D interval certificate, the local lower-bound methods, and the
global certification theorem.  All interval-certification statements
refer to outward-rounded interval arithmetic and to the archived
certificate data and verifier cited in the main manuscript.  No
statistical analysis, statistical sampling, or randomized numerical
experiment is used.

The logical dependencies are as follows.  Sections~S1--S2 formulate the
Lebesgue problem and reduce it to width-one constant-width bodies.
Sections~S3--S6 build the Reuleaux-type hierarchy and prove its
exactness and quantitative convergence.  Sections~S7--S14 then use a
low-order regular finite subtest from this hierarchy, reduce its
placements to a five-parameter domain, and verify the remaining finite
certificate.  The interval certificate is a deterministic verification
of archived rational boxes and finite witness data: the verifier
reconstructs constants and domains, checks the analytic and collar
reductions, verifies exact coverage of the certified domains, and
recomputes a rigorous local lower bound on every accepted leaf box.

\begin{center}
\begin{tabular}{p{0.17\textwidth}p{0.70\textwidth}}
\(\mathcal W_1\) & planar bodies of constant width \(1\)\\
\(\mathcal R_n\) & width-one Reuleaux-type bodies with \(n\) circular arcs\\
\(\mathcal A_M\) & \(M\)-th finite-arc hierarchy level,
\(\{D\}\cup\bigcup_{m=1}^M\mathcal R_{2m+1}\)\\
\(\Lambda_M\) & variational lower-bound value \(\Lambda(\mathcal A_M)\)\\
\(\mathcal F_{3,5}^{\rm R}\) & regular finite Reuleaux test family
\(\{D,R_3,R_5\}\)\\
\(L_{3,5}^{\rm R}\) & finite-test value \(\Lambda(\mathcal F_{3,5}^{\rm R})\)
\end{tabular}
\end{center}

\begin{center}
\begin{tabular}{p{0.10\textwidth}p{0.78\textwidth}}
S1 & Problem formulation\\
S2 & Reduction to constant-width bodies\\
S3 & Reuleaux-type constant-width test families\\
S4 & Variational lower-bound hierarchy\\
S5 & Exactness of the variational hierarchy\\
S6 & Quantitative convergence of the hierarchy\\
S7 & Organization of the certified finite-test proof\\
S8 & Finite Reuleaux test family\\
S9 & Complete five-parameter normalization\\
S10 & Analytic reduction to a bounded parameter domain\\
S11 & Collar certificates and reduction to the core box\\
S12 & Core5D interval certificate\\
S13 & Local lower-bound methods\\
S14 & Interval certification and the certified \(0.834\) theorem
\end{tabular}
\end{center}

\section{Problem Formulation}
\label{sec:problem}

We write $\mathcal E(2)$ for the full group of Euclidean isometries of the
plane, including reflections.  Throughout, area means two-dimensional
Lebesgue measure.  For a nonempty compact set $K\subset\R^2$, its
diameter is
\[
\operatorname{diam}(K)
=
\sup\{|x-y|:x,y\in K\}.
\]

A compact convex set $U\subset\R^2$ is called a \emph{universal
cover} for planar sets of diameter at most one if, for every nonempty
compact set $K\subset\R^2$ with
\[
\operatorname{diam}(K)\le 1,
\]
there exists an isometry $g\in\mathcal E(2)$ such that
\[
gK\subset U.
\]

The Lebesgue universal covering constant is defined as
\[
\aLeb
=
\inf\left\{
\operatorname{area}(U):
\begin{array}{l}
U\subset\R^2 \text{ is a compact convex universal cover}\\
\text{for all planar sets of diameter at most one}
\end{array}
\right\}.
\]

That is, $\aLeb$ is the infimum of the areas of planar convex
sets that contain a congruent copy of every nonempty compact set of
diameter at most $1$.

Throughout the paper, a \emph{convex body} means a compact convex
subset of $\R^2$ with nonempty interior.  The next step is to
replace the class of arbitrary nonempty compact sets of diameter at most
one by the geometrically more structured class of constant-width bodies.  This
reduction is carried out in the following section and provides the
starting point for the variational hierarchy.

\section{Reduction to Constant-Width Bodies}
\label{sec:constant-width}

The original formulation of Lebesgue's universal covering problem can be
reduced to the geometry of planar bodies of constant width.  This
classical reduction is the entry point for the Reuleaux-type hierarchy,
which is built from structured subfamilies of constant-width bodies.

It is enough to consider convex test sets.  If
$K\subset\R^2$ is nonempty and compact, then
\[
K\subseteq \operatorname{conv}(K),
\]
and
\[
\operatorname{diam}(\operatorname{conv}(K))
=
\operatorname{diam}(K).
\]
Hence any set that contains a congruent copy of every nonempty compact convex
set of diameter at most $1$ also contains a congruent copy of every
nonempty compact set of diameter at most $1$.  We may therefore restrict
attention to the class
\[
\mathcal K_1
:=
\{K\subset\R^2:
K\text{ is nonempty, compact, and convex, and }
\operatorname{diam}(K)\le 1\}.
\]

For a nonempty compact convex set $K$, let
\[
h_K(u)
=
\max_{x\in K}\langle x,u\rangle,
\qquad u\in S^1,
\]
denote its support function.  The width of $K$ in the direction $u$ is
\[
w_K(u)
=
h_K(u)+h_K(-u).
\]
We say that $K$ has \emph{constant width $1$} if
\[
w_K(u)=1
\qquad
\text{for every }u\in S^1.
\]

Let
\[
\mathcal W_1
:=
\{W\subset\R^2:
W\text{ is a convex body of constant width }1\}
\]
be the family of planar constant-width bodies of width $1$.

The following standard completion result is the key reduction; see, for
example, classical treatments of complete convex bodies and bodies of
constant width in planar convex geometry
~\cite{BonnesenFenchel1934,Eggleston1958,ChakerianGroemer1983,Schneider2014}.

\begin{lemma}\label{lem:completion}
For every $K\in\mathcal K_1$, there exists $W_K\in\mathcal W_1$ such
that
\[
K\subseteq W_K.
\]
\end{lemma}

\begin{proof}
Let $\delta=\operatorname{diam}(K)\le 1$.  If \(\delta=0\), then
\(K=\{p\}\) for some point \(p\).  The disk centered at \(p\) with
radius \(1/2\) has constant width \(1\), contains \(K\), and belongs to
\(\mathcal W_1\).  This proves the claim in this case.

Assume now that \(\delta>0\).  By the classical completion theorem for
planar compact convex sets of positive diameter, \(K\) is contained in a
complete convex body \(C\) of diameter \(\delta\).  In the plane,
complete convex bodies are precisely convex bodies of constant width
equal to their diameter
~\cite{BonnesenFenchel1934,ChakerianGroemer1983,Schneider2014}, so \(C\) has constant
width \(\delta\).

If $\delta=1$, we may take $W_K=C$.  If $\delta<1$, let
\[
r=\frac{1-\delta}{2}
\]
and define the parallel body
\[
W_K=C+rD_0,
\]
where
\[
D_0:=\{x\in\R^2: |x|\le 1\}
\]
is the unit disk centered at the origin.  Adding the disk $rD_0$
increases every width by $2r$, hence $W_K$ has constant width
\[
\delta+2r
=
\delta+(1-\delta)
=
1.
\]
Moreover \(W_K\in\mathcal W_1\), and \(K\subseteq C\subseteq W_K\), as
claimed.
\end{proof}

Define the corresponding constant-width covering value
\[
a_{\mathrm{cw}}
:=
\inf\left\{
\operatorname{area}(U):
\begin{array}{l}
U\subset\R^2 \text{ is compact and convex, and for every }W\in\mathcal W_1,\\
\text{there exists }g\in\mathcal E(2)\text{ with }gW\subset U
\end{array}
\right\}.
\]

The original Lebesgue constant is exactly this constant-width covering
value.

\begin{proposition}\label{prop:cw-equal}
One has
\[
a_{\mathrm{cw}}=\aLeb.
\]
\end{proposition}

\begin{proof}
Since every $W\in\mathcal W_1$ has diameter $1$, every universal cover
for planar sets of diameter at most one is in particular a cover for all
members of $\mathcal W_1$.  Therefore
\[
a_{\mathrm{cw}}\le \aLeb.
\]

Conversely, suppose $U$ is a compact convex set containing a congruent
copy of every $W\in\mathcal W_1$.  Let $K\in\mathcal K_1$.  By
Lemma~\ref{lem:completion}, there exists $W_K\in\mathcal W_1$ such
that
\[
K\subseteq W_K.
\]
Since $U$ covers $W_K$, there exists $g\in\mathcal E(2)$ such that
\[
gW_K\subset U.
\]
Then
\[
gK\subseteq gW_K\subset U.
\]
The set $U$ therefore covers every member of $\mathcal K_1$, and hence
every nonempty compact set of diameter at most $1$.  Thus
\[
\aLeb\le a_{\mathrm{cw}}.
\]

Combining the two inequalities gives
\[
a_{\mathrm{cw}}=\aLeb.
\]
\end{proof}

Lebesgue's universal covering problem may therefore be studied through
the family $\mathcal W_1$ of constant-width bodies.  Henceforth
we identify the constant-width covering value $a_{\mathrm{cw}}$ with
$\aLeb$ and work with $\mathcal W_1$ as the ambient test family.  The
next section constructs an increasing variational hierarchy by replacing
the full family $\mathcal W_1$ with structured Reuleaux-type
subfamilies.

\section{Reuleaux-Type Constant-Width Test Families}
\label{sec:reuleaux}

After the constant-width reduction, the ambient test family is
$\mathcal W_1$, the planar bodies of constant width $1$.  The variational
hierarchy will use the following structured subfamilies.

Let
\[
D
=
\{x\in\R^2: |x|\le 1/2\}
\]
be the disk of radius $1/2$.  It is the most symmetric member of
$\mathcal W_1$.  The disk plays the role of the infinite-arc limiting
member and is included explicitly in the hierarchy.

For a constant-width body $K\in\mathcal W_1$, write
\[
h_K(\theta)=h_K((\cos\theta,\sin\theta))
\]
for its support function in normal-angle coordinates, and define its
curvature-radius measure by
\[
\mu_K=h_K+h_K''
\]
in the sense of distributions.  For bodies of constant width \(1\), the
curvature-measure identity
\[
\mu_K+\tau_\pi\mu_K=d\theta
\]
holds, where \(\tau_\pi\) denotes push-forward under
\(\theta\mapsto\theta+\pi\).  This is obtained by applying
\(1+d^2/d\theta^2\) to the width identity
\(h_K(\theta)+h_K(\theta+\pi)=1\).  Since both terms on the left are
nonnegative measures, this identity implies \(\mu_K\ll d\theta\).  We
therefore write
\[
d\mu_K(\theta)=\rho_K(\theta)\,d\theta.
\]
The density satisfies the antipodal relation
\[
\rho_K(\theta)+\rho_K(\theta+\pi)=1
\]
for a.e. \(\theta\).

For an odd integer \(n\ge3\), let \(\mathcal R_n\) denote the family of
constant-width bodies \(K\in\mathcal W_1\) whose curvature-radius
density \(\rho_K\) is a \(0/1\)-valued step function and for which the
set
\[
\{\theta\in\R/2\pi\Z:\rho_K(\theta)=1\}
\]
has exactly \(n\) connected components.  Step functions are always
understood modulo equality a.e. on \(S^1=\R/2\pi\Z\).  When counting
components, we choose a finite-interval representative, delete
zero-length intervals, and merge adjacent intervals; the count is then
taken in the circle topology.  In this density description the
antipodal relation above holds, and the usual closing conditions are
\[
\int_0^{2\pi}\rho_K(\theta)\cos\theta\,d\theta=0,
\qquad
\int_0^{2\pi}\rho_K(\theta)\sin\theta\,d\theta=0.
\]
Geometrically, the intervals on which \(\rho_K=1\) correspond to
circular boundary arcs of radius \(1\), while intervals on which
\(\rho_K=0\) correspond to vertex normal cones.

For fixed \(n\), this is a finite-dimensional geometric family.  After
quotienting translations, a member is described by finitely many
arc-transition angles on \(S^1\), together with the antipodal relation,
the two closing equations, and the cyclic ordering and nondegeneracy
conditions.  Rotations act by adding a common angle to all transition
angles.  Equivalently, \(\mathcal R_n\) is a finite-parameter
Reuleaux-type family.

The following standard converse justifies this density-based
definition.

\begin{lemma}[Finite-step curvature densities]\label{lem:finite-step-density}
Let \(\rho:S^1\to\{0,1\}\) be a finite step function, understood up to
a.e. equality, and suppose that
\[
\rho(\theta)+\rho(\theta+\pi)=1
\quad\text{for a.e. }\theta.
\]
Assume also the closing conditions
\[
\int_0^{2\pi}\rho(\theta)\cos\theta\,d\theta=0,
\qquad
\int_0^{2\pi}\rho(\theta)\sin\theta\,d\theta=0.
\]
Then there exists a compact convex body \(K\), unique up to translation,
whose support function satisfies
\[
h_K+h_K''=\rho
\]
in the sense of distributions, and \(K\) has constant width \(1\).
If the set \(\{\rho=1\}\) has an odd number \(n\ge3\) of nondegenerate
connected components in the sense described above, then
\(K\in\mathcal R_n\).
Conversely, every \(K\in\mathcal R_n\) arises from such a density.
\end{lemma}

\begin{proof}
The measure \(\rho(\theta)\,d\theta\) is nonnegative and, by the
antipodal identity, has total mass \(\pi\).  Since it is absolutely
continuous with respect to Lebesgue measure, it is not concentrated on
a pair of antipodal points.  The two closing conditions are precisely
the vanishing first moment condition for the surface-area measure in
Minkowski's existence theorem.  Hence the planar Minkowski theorem in
support-function form~\cite{Schneider2014} gives a compact convex body,
unique up to translation, with curvature-radius measure
\(\rho(\theta)\,d\theta\), equivalently \(h_K+h_K''=\rho\).

Let \(w(\theta)=h_K(\theta)+h_K(\theta+\pi)\).  Then
\[
w+w''
=
\rho(\theta)+\rho(\theta+\pi)
=1.
\]
Because \(w\) is \(\pi\)-periodic, the first-harmonic terms in the
general solution of \(w+w''=1\) vanish, so \(w\equiv1\).  Thus \(K\)
has constant width \(1\).  On intervals where \(\rho=1\), the boundary
has curvature radius \(1\); where \(\rho=0\), the boundary contribution
is a vertex normal cone.  Therefore the nondegenerate components of
\(\{\rho=1\}\) are exactly the circular arcs of the Reuleaux-type
boundary.  This proves the asserted membership in \(\mathcal R_n\), and
the converse follows immediately from the definition of
\(\mathcal R_n\).
\end{proof}

For an arc-complexity cutoff \(q\ge3\), define
\[
\mathcal R_{\le q}
:=
\bigcup_{\substack{3\le n\le q\\ n\ \mathrm{odd}}}\mathcal R_n.
\]
The notation \(\mathcal R_{\le q}\) uses \(q\) as an arc-complexity cutoff,
whereas the hierarchy level \(\mathcal A_M\) defined below uses \(M\)
as a level index and is normalized so that its largest odd arc count is
\(2M+1\).  In particular, \(\mathcal R_3\) contains Reuleaux triangles
of width \(1\), while \(\mathcal R_5\) contains Reuleaux pentagons of
width \(1\).

Throughout the paper, calligraphic symbols such as $\mathcal R_n$
denote families of Reuleaux-type bodies.  In contrast, ordinary symbols
such as $R_3$ and $R_5$ are reserved for the fixed regular
representatives used in the finite test model of Part~II.

We shall use the disk $D$ together with these finite-arc families to
define the hierarchy.  The hierarchy index \(M\) counts the odd
Reuleaux levels rather than the maximum number of arcs.  For
\(M=1,2,3,\ldots\), we set
\[
\mathcal A_M
=
\{D\}\cup\bigcup_{m=1}^{M}\mathcal R_{2m+1}
=
\{D\}\cup\mathcal R_3\cup\mathcal R_5\cup\cdots\cup\mathcal R_{2M+1}.
\]
Using the cutoff notation introduced above,
\[
\mathcal A_M=\{D\}\cup\mathcal R_{\le 2M+1}.
\]

Here and below, the calligraphic symbol $\mathcal A_M$ denotes a test
family.  The ordinary symbol $A(z)$ introduced later denotes the area of
a finite-dimensional configuration.

The family $\mathcal A_M$ is a structured subfamily of $\mathcal W_1$.
As $M$ increases by one, the next odd Reuleaux level is added, and the
associated covering constraints become stronger.  These families provide
a finite-arc approximation to the full class $\mathcal W_1$ and, at each
fixed level, a finite-dimensional variational lower-bound problem from
which regular finite subtests can be extracted for certified computation.
In particular, the low-order truncation
\[
\{D\}\cup \mathcal R_3\cup\mathcal R_5
\]
will later lead to the finite-dimensional convex-hull problem associated
with the regular finite test family $\mathcal F_{3,5}^{\rm R}=\{D,R_3,R_5\}$.

\section{A Variational Lower-Bound Hierarchy}
\label{sec:variational-hierarchy}

The Reuleaux-type test families give variational lower bounds through the
following functional.

For any subfamily $\mathcal F\subseteq\mathcal W_1$, define
\[
\Lambda(\mathcal F)
=
\inf\left\{
\operatorname{area}(U):
\begin{array}{l}
U\subset\R^2 \text{ is compact and convex, and for every }K\in\mathcal F,\\
\text{there exists }g_K\in\mathcal E(2)\text{ such that }g_KK\subset U
\end{array}
\right\}.
\]
The value $\Lambda(\mathcal F)$ is the smallest possible area, in the
infimum sense, of a convex set that contains a congruent copy of every
member of $\mathcal F$.

For finite test families this definition has a useful equivalent form.
If
\[
\mathcal F=\{K_1,\ldots,K_N\},
\]
then
\[
\Lambda(\mathcal F)
=
\inf_{g_1,\ldots,g_N\in\mathcal E(2)}
\operatorname{area}
\operatorname{conv}
\left(
\bigcup_{j=1}^N g_jK_j
\right).
\]
Indeed, for any fixed placements \(g_1K_1,\ldots,g_NK_N\), the convex
hull
\[
\operatorname{conv}\left(\bigcup_{j=1}^N g_jK_j\right)
\]
is an admissible compact convex set for the family \(\mathcal F\), so
\(\Lambda(\mathcal F)\) is no larger than the right-hand side above.
Conversely, if \(U\) is admissible for \(\mathcal F\), then for each
\(j\) there exists \(g_j\in\mathcal E(2)\) such that \(g_jK_j\subset U\).
Since \(U\) is convex,
\[
\operatorname{conv}\left(\bigcup_{j=1}^N g_jK_j\right)
\subset U,
\]
and therefore
\[
\operatorname{area}
\operatorname{conv}\left(\bigcup_{j=1}^N g_jK_j\right)
\le \operatorname{area}(U).
\]
Taking the infimum over all admissible \(U\) gives the reverse
inequality.

Apply this lower-bound functional to the Reuleaux-type families
\[
\mathcal A_M
=
\{D\}\cup\bigcup_{m=1}^{M}\mathcal R_{2m+1},
\]
defined in the previous section.  The $M$-th variational lower-bound
value is
\[
\Lambda_M
:=
\Lambda(\mathcal A_M).
\]

In expanded form,
\[
\Lambda_M
=
\inf\left\{
\operatorname{area}(U):
\begin{array}{l}
U\subset\R^2 \text{ is compact and convex, and }U\\
\text{contains a congruent copy of every }K\in\mathcal A_M
\end{array}
\right\}.
\]

The values form a hierarchy of increasingly constrained covering problems.
Since
\[
\mathcal A_M\subseteq \mathcal W_1,
\]
every universal cover for all constant-width bodies also covers every
member of $\mathcal A_M$.  By Proposition~\ref{prop:cw-equal}, this
immediately gives
\[
\Lambda_M\le \aLeb
\qquad
\text{for every }M.
\]

The hierarchy is also monotone.

\begin{proposition}\label{prop:monotone}
For $M_1\le M_2$, one has
\[
\Lambda_{M_1}\le \Lambda_{M_2}.
\]
Consequently,
\[
\Lambda_1\le \Lambda_2\le\Lambda_3\le\cdots\le \aLeb.
\]
\end{proposition}

\begin{proof}
If $M_1\le M_2$, then
\[
\bigcup_{m=1}^{M_1}\mathcal R_{2m+1}
\subseteq
\bigcup_{m=1}^{M_2}\mathcal R_{2m+1},
\]
and hence
\[
\mathcal A_{M_1}\subseteq \mathcal A_{M_2}.
\]
A convex set that covers every member of $\mathcal A_{M_2}$ also
covers every member of $\mathcal A_{M_1}$.  Therefore the admissible
class defining $\Lambda_{M_2}$ is contained in the admissible class
defining $\Lambda_{M_1}$.  Taking infima gives
\[
\Lambda_{M_1}\le \Lambda_{M_2}.
\]
\end{proof}

The sequence \(\{\Lambda_M\}\) is a monotone hierarchy of finite-dimensional,
finite-complexity variational lower-bound problems for \(\aLeb\).
Finite subtests of these levels can be extracted for certified
lower-bound computations, as carried out in Part~II.  The main
theoretical question is whether the hierarchy loses any information in
the limit; the exactness theorem answers this question.

\section{Exactness of the Variational Hierarchy}
\label{sec:exactness}

The preceding section constructed a monotone sequence of lower-bound
values
\[
\Lambda_M=\Lambda(\mathcal A_M),
\qquad
\mathcal A_M
=
\{D\}\cup\bigcup_{m=1}^{M}\mathcal R_{2m+1},
\]
satisfying
\[
\Lambda_M\le \aLeb.
\]
No information is lost in the limit: the hierarchy
converges exactly to the Lebesgue universal covering constant.

We shall use the Hausdorff distance $d_H$ on compact subsets of $\R^2$.
Since all covering constraints are invariant under Euclidean motions, the
approximation estimates are stated in the corresponding congruence-invariant
shape distance
\[
d_{\mathcal E}(K,L)
:=
\inf_{q\in\mathcal E(2)} d_H(K,qL).
\]
In the support-function estimates below we fix the orientation produced by
the curvature-radius construction and remove translations by imposing
vanishing first harmonics.  This yields an ordinary Hausdorff estimate for
one representative, and hence an estimate in $d_{\mathcal E}$.  Whenever
an ordinary Hausdorff convergence statement is needed,
we choose Euclidean representatives realizing, or arbitrarily nearly
realizing, this infimum.  This convention only concerns the placement of test
bodies; it does not change any covering value, because each test body may be
placed by an arbitrary isometry inside the cover.

\subsection{A bang-bang selection tool}

The finite-dimensional purification input used below is the following.

\begin{lemma}[Bang-bang moment selection]\label{lem:bang-bang}
Let $I=[a,b]$, let $0\le r\le1$ be a measurable function on $I$, and
let
\[
\phi_1,\ldots,\phi_q\in L^1(I).
\]
Then there exists a measurable set $E\subset I$ such that
\[
\int_E \phi_\ell(\theta)\,d\theta
=
\int_I r(\theta)\phi_\ell(\theta)\,d\theta,
\qquad
\ell=1,\ldots,q.
\]
\end{lemma}

\begin{proof}
This is the standard purification principle for nonatomic measures.
We include the argument because it is the mechanism that converts a
general curvature-radius density into a Reuleaux-type $0/1$ density.

Let
\[
\nu(A)=
\left(
\int_A\phi_1(\theta)\,d\theta,\ldots,
\int_A\phi_q(\theta)\,d\theta
\right)\in\R^q
\]
be the vector measure generated by the functions $\phi_\ell$.  Since
Lebesgue measure on an interval is nonatomic, Lyapunov's convexity
theorem implies that the range
\[
\mathcal R_\nu=\{\nu(A):A\subset I\text{ measurable}\}
\]
is compact and convex~\cite{Lyapunov1940}.  The vector
\[
v=
\left(
\int_I r(\theta)\phi_1(\theta)\,d\theta,\ldots,
\int_I r(\theta)\phi_q(\theta)\,d\theta
\right)
\]
belongs to this convex range.  One way to see this is to write
\[
\int_I r(\theta)\phi_\ell(\theta)\,d\theta
=
\int_0^1\int_{\{r>t\}}\phi_\ell(\theta)\,d\theta\,dt
\]
for each $\ell$; hence $v$ is a barycenter of points of
$\mathcal R_\nu$, and therefore belongs to $\mathcal R_\nu$ by
convexity.  Thus there exists a measurable set $E\subset I$ with
$\nu(E)=v$, which is exactly the asserted moment identity.
\end{proof}

\begin{theorem}[Affine finite-switch Hobby--Rice selection]\label{thm:affine-hr-selection}
Let \(I\) be a compact interval, let \(0\le r\le1\) be measurable, and
let \(V\subset L^1(I)\) be finite-dimensional.  Assume that \(I\) is
partitioned into finitely many intervals and that, on each partition
cell, every element of \(V\) is represented by a function from a fixed
finite-dimensional extended Chebyshev space.  Let \(\mathcal P\) denote
this partition.  Then there exists a set
\(E\subset I\), equal up to a null set to a finite union of intervals,
such that
\[
\int_I(\mathbf 1_E-r)v\,d\theta=0
\qquad
\text{for every }v\in V.
\]
Moreover, \(E\) may be chosen so that its switching points are contained
in the partition endpoints and in the zero set of one nonzero switching
function from \(V\).  Consequently, if every nonzero element of the
local Chebyshev space has at most \(d\) zeros on each partition cell,
then the number of intervals in \(E\) is bounded by \(C|\mathcal P|+C'\),
where \(|\mathcal P|\) is the number of partition cells and \(C,C'\)
depend only on \(d\) and on the finite number of global moment functions.
\end{theorem}

\begin{proof}
This is the affine finite-switch form of the one-dimensional
Hobby--Rice purification theorem.  In the special case \(r\equiv1/2\)
it is the usual Hobby--Rice theorem for sign functions.  The affine
form allows an arbitrary density \(0\le r\le1\), equivalently an
arbitrary target point in the range of the nonatomic vector measure
generated by \(V\); it is the one-dimensional finite-switch
purification used in relaxed-control and nonatomic vector-measure
settings~\cite{Lyapunov1940,DvoretzkyWaldWolfowitz1951,HobbyRice1965,Warga1972}.
This finite-switch selection principle is the only external purification
input used in the exactness proof; below it is used only through the
mesh-moment form of Proposition~\ref{prop:finite-switch-purification}.
The final switch-count assertion follows from the Chebyshev zero-count
bound on each partition cell, together with the partition endpoints.
\end{proof}

\begin{proposition}[Finite-switch affine Hobby--Rice purification for mesh moments]\label{prop:finite-switch-purification}
Let \(I_0,\ldots,I_{N-1}\) be a partition of \([0,\pi]\) into
intervals, and let \(c_j\) be the midpoint of \(I_j\).  For every
measurable function \(0\le r\le1\) on \([0,\pi]\), there is a measurable
set \(E_N\subset[0,\pi]\), equal up to a null set to a union of at most
\(C_0N+C_1\) intervals for absolute constants \(C_0,C_1\), such that
\(\chi_N=\mathbf 1_{E_N}\) satisfies
\[
\int_{I_j}(\chi_N-r)\,d\theta=0,
\qquad j=0,\ldots,N-1,
\]
and the two closing moment identities
\[
\int_0^\pi(\chi_N-r)\cos\theta\,d\theta
=
\int_0^\pi(\chi_N-r)\sin\theta\,d\theta
=0.
\]
Moreover, the same conclusion holds with the additional local
first-moment identities
\[
\int_{I_j}(\theta-c_j)(\chi_N-r)\,d\theta=0,
\qquad j=0,\ldots,N-1.
\]
\end{proposition}

\begin{proof}
For the qualitative statement, take the \(N\) cell-mass functions
\(\mathbf 1_{I_j}\) and the two closing functions \(\cos\theta\) and
\(\sin\theta\).  This gives \(N+2\) imposed moments.  For the
quantitative statement, also include the \(N\) local first-moment
functions \((\theta-c_j)\mathbf 1_{I_j}\), giving \(2N+2\) imposed
moments.  Apply Theorem~\ref{thm:affine-hr-selection} to the span of
these moment functions.  It gives a selector
\(\chi_N=\mathbf 1_{E_N}\) with the same prescribed moments as \(r\),
and its switching set is contained in the partition endpoints together
with the zero set of one nonzero switching function from this span.

The switch count is linear in \(N\), because the nonsmooth basis
functions are supported on the \(N\) mesh cells.  On each mesh cell the
switching function has the form
\[
a_j+b_j(\theta-c_j)+A\cos\theta+B\sin\theta
\]
with constants \(a_j,b_j,A,B\), where \(b_j=0\) in the qualitative
version without local first moments.  The four functions
\(1,\theta,\cos\theta,\sin\theta\) form an extended Chebyshev system on
any interval of length at most \(\pi\): their Wronskian is
\[
\det
\begin{pmatrix}
1&\theta&\cos\theta&\sin\theta\\
0&1&-\sin\theta&\cos\theta\\
0&0&-\cos\theta&-\sin\theta\\
0&0&\sin\theta&-\cos\theta
\end{pmatrix}
=1.
\]
Hence any nonzero restriction has at most three zeros on a mesh cell,
counting multiplicity.  Therefore its sign decomposition has at most a
fixed number of intervals on each cell, and over \(N\) cells the selector
has at most \(C_0N+C_1\) sign intervals, for absolute constants
\(C_0,C_1\).

Changing the selector on a null set and merging adjacent intervals makes
\(E_N\) a finite union of at most \(C_0N+C_1\) intervals without changing
any moment identity.
\end{proof}

\subsection{Density of finite-arc constant-width bodies}

The first ingredient is the density of finite-arc Reuleaux-type bodies
in the space of constant-width bodies.

\begin{lemma}\label{lem:density}
For every $W\in\mathcal W_1$ and every $\varepsilon>0$, there exists
a hierarchy level $M$ and a body $R\in\mathcal A_M$ such that
\[
d_{\mathcal E}(W,R)<\varepsilon.
\]
In closure notation,
\[
\overline{\bigcup_{M\ge1}\mathcal A_M}^{\,d_{\mathcal E}}
=
\mathcal W_1.
\]
\end{lemma}

\begin{proof}
Let $W\in\mathcal W_1$, and let $h=h_W$ be its support function.  We
normalize $h$ by removing its first-harmonic part; this only fixes the
translation of $W$.  Let
\[
\mu_W=h+h''
\]
be the curvature-radius measure of $W$.  Since $W$ has constant width
$1$, the identity
\[
h(\theta)+h(\theta+\pi)=1
\]
implies, in the sense of measures,
\[
\mu_W(A)+\mu_W(A+\pi)=|A|
\]
for every Borel set $A\subset S^1$.  Hence $\mu_W$ is absolutely
continuous with respect to Lebesgue measure, and
\[
d\mu_W(\theta)=\rho(\theta)\,d\theta,
\qquad
0\le \rho\le1,
\qquad
\rho(\theta)+\rho(\theta+\pi)=1
\]
for almost every $\theta$.  The standard closing condition for
curvature-radius measures gives
\[
\int_0^{2\pi}\rho(\theta)\cos\theta\,d\theta=0,
\qquad
\int_0^{2\pi}\rho(\theta)\sin\theta\,d\theta=0.
\]

Construct finite-switch $0/1$ densities approximating $\rho$ as follows.
For each $N\ge1$, partition $[0,\pi]$ into intervals
\[
\theta_j=\frac{j\pi}{N},
\qquad
I_{N,j}=[\theta_j,\theta_{j+1}],
\qquad
j=0,\ldots,N-1,
\]
of equal length $\pi/N$.  Apply
Proposition~\ref{prop:finite-switch-purification} in its qualitative
mesh-moment form to the partition \(\{I_{N,j}\}\) and to
\(r=\rho|_{[0,\pi]}\).
We obtain a finite union of intervals $E_N\subset[0,\pi]$ such that
the function
\[
\chi_N=\mathbf 1_{E_N}
\]
satisfies
\[
\int_{I_{N,j}}(\chi_N-\rho)\,d\theta=0
\qquad (j=0,\ldots,N-1),
\]
and
\[
\int_0^\pi(\chi_N-\rho)\cos\theta\,d\theta=
\int_0^\pi(\chi_N-\rho)\sin\theta\,d\theta=0.
\]
Extend $\chi_N$ to $S^1$ by
\[
\chi_N(\theta+\pi)=1-\chi_N(\theta).
\]
Then $\chi_N$ is $0/1$-valued, has finitely many switching points, and
satisfies
\[
\chi_N(\theta)+\chi_N(\theta+\pi)=1.
\]
The half-circle first-moment equalities and the same identities for
$\rho$ imply the full-circle closing conditions
\[
\int_0^{2\pi}\chi_N(\theta)\cos\theta\,d\theta=0,
\qquad
\int_0^{2\pi}\chi_N(\theta)\sin\theta\,d\theta=0.
\]

The measure \(\chi_N(\theta)\,d\theta\) is nonnegative and has total mass
\[
\int_0^{2\pi}\chi_N(\theta)\,d\theta=\pi .
\]
It is not concentrated on a pair of antipodal points, and it has
vanishing first moment by the preceding identities.
By the planar Minkowski existence theorem in support-function
form~\cite{Schneider2014}, there is a convex body $R_N$, unique up to
translation, whose support function $h_N$ satisfies
\[
h_N+h_N''=\chi_N.
\]
The antipodal identity for $\chi_N$ implies that $R_N$ has constant
width $1$.  Indeed, if
\[
w_N(\theta)=h_N(\theta)+h_N(\theta+\pi),
\]
then $w_N+w_N''=1$ and $w_N$ is $\pi$-periodic, hence $w_N\equiv1$.

Because $\chi_N$ is a finite-switch $0/1$ density, after choosing its
interval representative up to null sets, merging adjacent intervals, and
deleting null intervals, the set \(\{\chi_N=1\}\) has finitely many
connected components.  The antipodal rule forces the number $n_N$ of
these components to be odd: along any half-turn from a nonswitching point
$\theta_0$ to $\theta_0+\pi$, the value of $\chi_N$ changes from $0$ to
$1$ or from $1$ to $0$, so the number of switches on that half-turn is
odd, and the opposite half-turn is its translate by $\pi$.  Thus the
number of circular-arc components is finite and odd.  The closing
condition excludes the one-arc
degenerate case: if there were only one component, the antipodal rule
would force it to be a semicircle up to null sets, and the integral of
$(\cos\theta,\sin\theta)$ over any semicircle is nonzero.  Hence
$n_N\ge3$.  On each component where $\chi_N=1$,
the boundary has curvature radius $1$, and on the complementary
components the curvature radius is $0$.  Hence
\[
R_N\in\mathcal R_{n_N}
\subset \mathcal A_{(n_N-1)/2}.
\]
Here \((n_N-1)/2\) is an integer hierarchy level because \(n_N\) is odd
and at least \(3\).

It remains to prove convergence.  Put
\[
f_N=\chi_N-\rho.
\]
For every continuous function $\varphi$ on $[0,\pi]$,
\[
\int_0^\pi f_N(\theta)\varphi(\theta)\,d\theta
=
\sum_{j=0}^{N-1}
\int_{I_{N,j}} f_N(\theta)
\bigl(\varphi(\theta)-\varphi(c_{N,j})\bigr)\,d\theta,
\]
where $c_{N,j}$ is the midpoint of $I_{N,j}$.  Since
$|f_N|\le1$ and the mesh tends to zero, the right-hand side tends to
zero by uniform continuity of $\varphi$.  Using
$f_N(\theta+\pi)=-f_N(\theta)$ gives weak convergence
\[
f_N\,d\theta \rightharpoonup 0
\]
on the full circle.
The imposed first-moment equalities also give
\[
\int_0^{2\pi} f_N(\theta)\cos\theta\,d\theta
=
\int_0^{2\pi} f_N(\theta)\sin\theta\,d\theta
=0,
\]
so \(f_N\) lies in the subspace on which the periodic Green operator for
\(1+d^2/d\theta^2\) is defined after translations have been normalized.

Let $G$ be the periodic Green kernel of $1+d^2/d\theta^2$ on the
subspace orthogonal to the first harmonics.  The kernel $G$ is
continuous, and
\[
h_N-h=\int_0^{2\pi}G(\theta-t)f_N(t)\,dt
\]
after the same no-first-harmonic normalization.  The translates
$G(\theta-\cdot)$ form a compact subset of $C(S^1)$, and the measures
$f_N\,d\theta$ have uniformly bounded total variation.  A finite-net
argument on this compact family of test functions upgrades the weak
convergence above to convergence uniform in \(\theta\).  Therefore
\[
\|h_N-h\|_{L^\infty(S^1)}\to0.
\]
Consequently
\[
d_{\mathcal E}(W,R_N)\le \|h_N-h\|_\infty\to0.
\]
Choosing $N$ large enough gives a body $R_N$ in some hierarchy level
$\mathcal A_M$ with $d_{\mathcal E}(W,R_N)<\varepsilon$.
\end{proof}

\begin{corollary}[Density along prescribed levels]\label{cor:density-along-levels}
For every $W\in\mathcal W_1$ and every sequence
$M_j\to\infty$, one may choose $R_j\in\mathcal A_{M_j}$ such that
\[
d_{\mathcal E}(R_j,W)\to 0.
\]
\end{corollary}

\begin{proof}
For each $k$ choose $S_k\in\mathcal A_{L_k}$ with
$d_{\mathcal E}(S_k,W)<1/k$.  Since $\mathcal A_M$ is increasing and
$M_j\to\infty$, for each large $j$ one can choose $k(j)\to\infty$ with
$L_{k(j)}\le M_j$ and then set $R_j=S_{k(j)}$.
For the finitely many remaining indices, choose any element of
\(\mathcal A_{M_j}\).  These choices do not affect the limit.
\end{proof}

This density statement is the geometric reason why the hierarchy can
recover the full Lebesgue constant in the limit.

\subsection{A compactness observation}

We shall also use a simple compactness fact for nearly optimal
admissible covers.  Since $D\in\mathcal A_M$, every admissible cover
for $\mathcal A_M$ contains some congruent copy of $D$.  Applying a
global isometry to the cover, we may assume that this copy is the disk
\[
D=\{x\in\R^2:|x|\le 1/2\}
\]
centered at the origin.

\begin{lemma}\label{lem:bounded}
Let $A>0$.  The family of compact convex sets $U\subset\R^2$
satisfying
\[
D\subset U,
\qquad
\operatorname{area}(U)\le A
\]
is uniformly bounded.
\end{lemma}

\begin{proof}
Let $p\in U$ and write $R=|p|$.  Since $D\subset U$ and $U$ is
convex, $U$ contains the convex hull of $p$ and $D$.  In particular,
after rotating coordinates so that $p=(R,0)$, $U$ contains the
triangle with vertices
\[
p,\qquad (0,1/2),\qquad (0,-1/2).
\]
This triangle has area $R/2$.  Hence
\[
\frac{R}{2}\le \operatorname{area}(U)\le A,
\]
so $R\le 2A$.  Since $p\in U$ was arbitrary,
$U\subset B(0,2A)$.  This proves uniform boundedness.
\end{proof}

By the Blaschke selection theorem, any sequence of such convex bodies
has a Hausdorff-convergent subsequence.

We shall also use the following closedness property.

\begin{lemma}\label{lem:closed}
Let $U_j\to U$ and $K_j\to K$ in Hausdorff distance, with all sets
nonempty and compact, and with $U_j,U$ convex.  Suppose that for each
$j$, there exists an isometry $g_j\in\mathcal E(2)$ such that
\[
g_jK_j\subset U_j.
\]
If the sets $U_j$ are uniformly bounded, then there exists an isometry
$g\in\mathcal E(2)$ such that
\[
gK\subset U.
\]
\end{lemma}

\begin{proof}
Write $g_jx=Q_jx+t_j$, where $Q_j\in O(2)$ and $t_j\in\R^2$.  Since
$O(2)$ is compact, a subsequence of $Q_j$ converges to some
$Q\in O(2)$.  Choose \(x_j\in K_j\).  Since \(K_j\to K\) in Hausdorff
distance and \(K\) is compact, the points \(x_j\) lie in a uniformly
bounded set.  Also \(Q_jx_j+t_j\in g_jK_j\subset U_j\), and the sets
\(U_j\) are uniformly bounded.  Hence the translations \(t_j\) are
bounded.  Passing to a further subsequence, $t_j\to t$.  Let
$g x=Qx+t$.  Hausdorff
convergence gives
\[
g_jK_j\to gK,
\qquad
U_j\to U.
\]
The inclusion $g_jK_j\subset U_j$ is closed under Hausdorff limits,
hence
\[
gK\subset U.
\]
\end{proof}

\subsection{Exactness theorem}

The main result of this part is the exactness theorem.

\begin{theorem}\label{thm:exactness}
The variational hierarchy is exact:
\[
\lim_{M\to\infty}\Lambda_M=\aLeb.
\]
\end{theorem}

\begin{proof}
By Proposition~\ref{prop:monotone}, the sequence $\{\Lambda_M\}$ is
monotone nondecreasing, and by Proposition~\ref{prop:cw-equal},
\[
\Lambda_M\le \aLeb
\qquad
\text{for every }M.
\]
Hence the limit
\[
L:=\lim_{M\to\infty}\Lambda_M
\]
exists and satisfies
\[
L\le \aLeb.
\]

It remains to prove the opposite inequality.

Fix $\eta>0$.  Choose a sequence $M_j\to\infty$.  For each $j$, let
$U_j$ be an admissible convex cover for $\mathcal A_{M_j}$ such that
\[
\operatorname{area}(U_j)
\le
\Lambda_{M_j}+\eta.
\]

Since $D\in\mathcal A_{M_j}$, we may apply a global isometry and
assume
\[
D\subset U_j.
\]
Moreover,
\[
\operatorname{area}(U_j)
\le
\aLeb+\eta
\]
for all $j$.  By Lemma~\ref{lem:bounded}, the sets $U_j$ are uniformly
bounded.  By the Blaschke selection theorem, after passing to a
subsequence we may assume that
\[
U_j\to U
\]
in Hausdorff distance for some compact convex set $U$.

We claim that $U$ contains a congruent copy of every
$W\in\mathcal W_1$.  Let $W\in\mathcal W_1$.  By
Corollary~\ref{cor:density-along-levels}, choose
\(R_j\in\mathcal A_{M_j}\) with \(d_{\mathcal E}(R_j,W)\to0\).  By the
definition of \(d_{\mathcal E}\), after passing to near-minimizing
isometries \(q_j\in\mathcal E(2)\), we may arrange that
\[
d_H(q_jR_j,W)\to 0.
\]
Because $U_j$ is admissible for $\mathcal A_{M_j}$, there exists an
isometry $g_j\in\mathcal E(2)$ such that
\[
g_jR_j\subset U_j.
\]
Set
\[
\widetilde R_j=q_jR_j,
\qquad
\widetilde g_j=g_jq_j^{-1}.
\]
Then $\widetilde R_j\to W$ in ordinary Hausdorff distance and
\[
\widetilde g_j\widetilde R_j=g_jR_j\subset U_j.
\]
Applying Lemma~\ref{lem:closed} with $K_j=\widetilde R_j$ and $K=W$, we
obtain an isometry $g\in\mathcal E(2)$ such that
\[
gW\subset U.
\]

Thus $U$ covers every member of $\mathcal W_1$.  By the constant-width
reduction proved in Section~\ref{sec:constant-width}, a compact convex
set with this property is a universal cover for planar sets of diameter
at most one.  Hence
\[
\aLeb\le \operatorname{area}(U).
\]

Area is continuous under Hausdorff convergence of planar compact convex
sets, equivalently by convergence of support functions, so
\[
\operatorname{area}(U)
=
\lim_{j\to\infty}\operatorname{area}(U_j)
\le
\lim_{j\to\infty}(\Lambda_{M_j}+\eta)
=
L+\eta.
\]
Therefore
\[
\aLeb\le L+\eta.
\]
Since $\eta>0$ was arbitrary, we obtain
\[
\aLeb\le L.
\]

Together with $L\le \aLeb$, this gives
\[
L=\aLeb.
\]
Thus
\[
\lim_{M\to\infty}\Lambda_M=\aLeb.
\]
\end{proof}

Theorem~\ref{thm:exactness} shows that the Reuleaux-type variational
hierarchy recovers the full Lebesgue universal covering constant in the
limit.  The convergence estimate below makes this statement quantitative.

\section{Quantitative Convergence of the Hierarchy}
\label{sec:convergence}

Sections~\ref{sec:variational-hierarchy} and~\ref{sec:exactness}
established that the Reuleaux-type hierarchy gives a monotone sequence
of lower bounds
\[
\Lambda_M\le \aLeb,
\qquad
\Lambda_M\uparrow \aLeb.
\]

The convergence rate is controlled by
\[
0\le \aLeb-\Lambda_M\le C M^{-2}.
\]

The first step estimates the congruence-invariant Hausdorff approximation
error of the Reuleaux hierarchy.  Define
\[
\delta_M
:=
\sup_{W\in\mathcal W_1}
\inf_{R\in\mathcal A_M}
d_{\mathcal E}(W,R),
\]
where $\mathcal W_1$ is the class of planar convex bodies of constant
width $1$, and
\[
\mathcal A_M
=
\{D\}\cup\bigcup_{m=1}^{M}\mathcal R_{2m+1}
\]
is the \(M\)-th Reuleaux-type hierarchy level.  We prove that
\[
\delta_M=O(M^{-2}).
\]
In fact, the same method also gives the sharp Hausdorff approximation
rate
\[
\delta_M=\Theta(M^{-2}).
\]
This identifies \(M^{-2}\) as the intrinsic Hausdorff approximation
scale of the finite-arc hierarchy.

Second, we transfer Hausdorff approximation of test bodies into area
convergence of the variational lower-bound values.  We prove that
there are constants $C_{\rm tr},C_{\rm St}>0$ such that
\[
\aLeb-\Lambda_M
\le
C_{\rm tr}\delta_M+C_{\rm St}\delta_M^2.
\]
Combining this transfer estimate with $\delta_M=O(M^{-2})$ gives
\[
0\le \aLeb-\Lambda_M\le C M^{-2}.
\]

The proof is organized as follows.  Section~\ref{sec:curvature}
defines the Reuleaux approximation error and recalls the
curvature-radius representation of constant-width bodies.
Section~\ref{sec:bang-bang} proves the bang-bang approximation
theorem, which gives the upper bound $\delta_M=O(M^{-2})$.
Section~\ref{sec:lower-scale} proves the matching lower bound
$\delta_M\ge cM^{-2}$.  Section~\ref{sec:transfer} proves the area
transfer lemma.  Section~\ref{sec:main-quant} combines these results
into the main quantitative convergence theorem.

The estimate
\[
\delta_M=\Theta(M^{-2})
\]
concerns Hausdorff approximation of constant-width bodies by finite-arc
Reuleaux-type bodies.  For the hierarchy values, the transfer argument
above gives the one-sided estimate
\[
0\le \aLeb-\Lambda_M\le C M^{-2}.
\]

\subsection{Reuleaux approximation error and curvature-radius
representation}
\label{sec:curvature}

We begin with the congruence-invariant Hausdorff approximation error of
the Reuleaux-type hierarchy:
\[
\delta_M
:=
\sup_{W\in\mathcal W_1}
\inf_{R\in\mathcal A_M}
d_{\mathcal E}(W,R).
\]
The quantity $\delta_M$ measures how well all constant-width bodies of
width $1$ can be approximated, up to Euclidean congruence, by the first $M$ odd Reuleaux levels, equivalently by finite-arc
Reuleaux-type bodies of arc complexity at most $2M+1$, together with
the disk $D$.

For compact convex sets $K,L\subset\R^2$ in a fixed common placement,
the Hausdorff distance can be expressed through support functions:
\[
d_H(K,L)
=
\|h_K-h_L\|_{L^\infty(S^1)}
\]
after the same origin has been fixed.  Translations add first harmonics
to the support function.  In the construction below we choose the
translating first harmonic so that the support-function difference has no
first-harmonic component.  This yields an ordinary Hausdorff estimate for
one placement of the approximant, and therefore an estimate for
$d_{\mathcal E}$.

Let $K\in\mathcal W_1$, and write its support function as
\[
h_K(\theta)
=
\max_{x\in K}\langle x,(\cos\theta,\sin\theta)\rangle.
\]
Since $K$ has constant width $1$,
\[
h_K(\theta)+h_K(\theta+\pi)=1.
\]

In the sense of distributions, define the curvature-radius measure
\[
\mu_K
:=
h_K+h_K''.
\]
For every planar convex body, $\mu_K$ is a nonnegative finite measure
on $S^1$.  The constant-width condition implies, again in the sense of
measures,
\[
\mu_K(A)+\mu_K(A+\pi)=|A|
\]
for every Borel set $A\subset S^1$, where $|A|$ denotes
one-dimensional Lebesgue measure on the circle.  Indeed, applying the
operator $1+\frac{d^2}{d\theta^2}$ to
\[
h_K(\theta)+h_K(\theta+\pi)=1
\]
gives exactly this relation.

It follows that $\mu_K$ is absolutely continuous with respect to
Lebesgue measure.  Hence
\[
d\mu_K(\theta)=\rho_K(\theta)\,d\theta
\]
for some measurable function $\rho_K$, and the measure identity above
gives
\[
0\le \rho_K(\theta)\le 1,
\qquad
\rho_K(\theta)+\rho_K(\theta+\pi)=1
\]
for almost every $\theta$.

Conversely, a measurable function $\rho$ satisfying
\[
0\le \rho\le 1,
\qquad
\rho(\theta)+\rho(\theta+\pi)=1,
\]
together with the closing conditions
\[
\int_0^{2\pi}\rho(\theta)\cos\theta\,d\theta=0,
\qquad
\int_0^{2\pi}\rho(\theta)\sin\theta\,d\theta=0,
\]
determines a constant-width convex body through the equation
\[
h+h''=\rho.
\]
Indeed, the nonnegative measure \(\rho(\theta)\,d\theta\) has total
mass \(\pi\) by the antipodal identity.  Since it is absolutely
continuous, it is not concentrated on a pair of antipodal points.  By
the planar Minkowski existence theorem in support-function form, this
measure, together with the closing conditions, determines a convex body
whose support function solves this equation.  The solution $h$ is unique
up to addition of first harmonics, which corresponds exactly to
translating the body.

For a Reuleaux-type finite-arc body, the curvature-radius density is
$0/1$-valued.  It equals $1$ on normal-angle intervals corresponding
to circular boundary arcs of radius $1$, and equals $0$ on the
complementary intervals corresponding to vertices.  Thus approximating
a general $W\in\mathcal W_1$ by Reuleaux-type bodies is equivalent to
approximating its density $\rho_W$ by $0/1$-valued functions
satisfying the antipodal constant-width constraint and the closing
conditions.

The next subsection proves that this can be done with support-function
error $O(M^{-2})$.

\subsection{Bang-bang approximation and the upper bound for \texorpdfstring{$\delta_M$}{delta M}}
\label{sec:bang-bang}

We prove the upper estimate
\[
\delta_M\le C M^{-2}.
\]
The proof uses the bang-bang moment selection lemma and its
finite-switch refinement,
Lemma~\ref{lem:bang-bang} and
Proposition~\ref{prop:finite-switch-purification}, already proved in
the exactness section.  In
the present subsection we impose one additional local first-moment
condition on each mesh interval.  That extra moment is what upgrades
the qualitative density argument to a second-order support-function
estimate.

\begin{lemma}[Green-kernel second-order estimate]\label{lem:green-estimate}
Let $[0,2\pi]$ be partitioned into intervals $I_j$ of length at most
$\Delta$, with midpoints $c_j$.  Suppose $f\in L^\infty(S^1)$ satisfies
$\|f\|_\infty\le 1$, has vanishing first harmonics,
\[
\int_0^{2\pi}f(t)\cos t\,dt=
\int_0^{2\pi}f(t)\sin t\,dt=0,
\]
and satisfies, on every interval $I_j$,
\[
\int_{I_j} f(t)\,dt=0,
\qquad
\int_{I_j} (t-c_j) f(t)\,dt=0.
\]
Let $u$ be the $2\pi$-periodic solution of
\[
u+u''=f
\]
with no first-harmonic component.  Then
\[
\|u\|_{L^\infty(S^1)}\le C_G\Delta^2,
\]
where $C_G$ is an absolute constant.
\end{lemma}

\begin{proof}
Let $G$ be the periodic Green kernel of $1+d^2/d\theta^2$ on the
subspace orthogonal to $\cos\theta$ and $\sin\theta$.  Then
\[
u(\theta)=\int_0^{2\pi}G(\theta-t)f(t)\,dt.
\]
The kernel $G$ is bounded and piecewise $C^2$, with only the standard
first-derivative jump at $t=\theta$ modulo \(2\pi\).  Fix \(\theta\).
After choosing a representative on \([0,2\pi]\), at most two partition
intervals meet this singular point modulo \(2\pi\), the second only when
the singular point lies at the periodic endpoint.  Call these intervals
exceptional.

On each nonexceptional interval \(I_j\), expand \(G(\theta-t)\) around
\(t=c_j\) to first order.  Such an interval lies in one smooth branch of
the kernel, where the second derivative is uniformly bounded.  The
zeroth and first order terms vanish by the two moment conditions on
\(I_j\).  The remaining contribution is bounded by
\[
C\Delta^2\int_{I_j}|f(t)|\,dt\le C\Delta^3.
\]
There are \(O(\Delta^{-1})\) nonexceptional intervals, so their total
contribution is \(O(\Delta^2)\).

For an exceptional interval \(I_j\), split \(I_j\) at the singular point
if it lies in the interior.  The kernel is Lipschitz on each resulting
piece, with a uniform Lipschitz constant; only its first derivative has
the standard jump.  Since \(\int_{I_j}f=0\), we may subtract the constant
value \(G(\theta-c_j)\) and write
\[
\int_{I_j}G(\theta-t)f(t)\,dt
=
\int_{I_j}\bigl(G(\theta-t)-G(\theta-c_j)\bigr)f(t)\,dt .
\]
Continuity at the singular point, together with the one-sided Lipschitz
bounds, gives the uniform estimate on the whole interval
\[
|G(\theta-t)-G(\theta-c_j)|\le C\Delta,
\]
and \(\|f\|_\infty\le1\), so this interval contributes at most
\(C\Delta |I_j|\le C\Delta^2\).  Since there are at most two exceptional
intervals, their total contribution is also \(O(\Delta^2)\).  The same
argument covers the cut point of the circle: only \(O(1)\) mesh cells are
affected by the Green-kernel singularity or by the choice of periodic
representative, and their total contribution remains \(O(\Delta^2)\).
Taking the supremum over \(\theta\) gives the stated estimate.
\end{proof}

Applying the lemma to constant-width bodies gives the upper approximation
rate.

\begin{proposition}[Upper Hausdorff approximation rate]\label{prop:upper-rate}
There exists a constant $C_{\rm app}>0$ such that
\[
\delta_M\le C_{\rm app}M^{-2}
\]
for all $M\ge1$.
\end{proposition}

\begin{proof}
Let $W\in\mathcal W_1$, and let $\rho=\rho_W$ be its curvature-radius
density.  We work first on the half-circle $[0,\pi]$.  Partition
$[0,\pi]$ into $N$ intervals by
\[
\theta_j=\frac{j\pi}{N},
\qquad
j=0,\ldots,N,
\]
and
\[
I_j=[\theta_j,\theta_{j+1}],
\qquad
j=0,\ldots,N-1,
\]
of equal length
\[
\Delta=\frac{\pi}{N}.
\]
Let $c_j$ be the midpoint of $I_j$.

Apply Proposition~\ref{prop:finite-switch-purification} in its
quantitative mesh-moment form to the partition \(\{I_j\}\) and to
\(r=\rho|_{[0,\pi]}\).  This mesh system has \(2N+2\) imposed moments,
so we obtain a measurable set $E_N\subset[0,\pi]$, chosen as a union of
\(O(N)\) intervals, whose characteristic function
\[
\chi_N=\mathbf 1_{E_N}
\]
satisfies the local moment identities
\[
\int_{I_j}(\chi_N-\rho)\,d\theta=0,
\]
\[
\int_{I_j}(\theta-c_j)(\chi_N-\rho)\,d\theta=0
\]
for \(j=0,\ldots,N-1\), and also the half-circle closing identities
\[
\int_0^\pi(\chi_N-\rho)\cos\theta\,d\theta=0,
\]
\[
\int_0^\pi(\chi_N-\rho)\sin\theta\,d\theta=0.
\]

Extend $\chi_N$ to the full circle by
\[
\chi_N(\theta+\pi)=1-\chi_N(\theta).
\]
Then
\[
\chi_N(\theta)+\chi_N(\theta+\pi)=1.
\]
The imposed half-circle first-harmonic identities imply the full-circle
closing conditions.  Indeed, the original density $\rho$ satisfies the
same antipodal identity and has vanishing full-circle first moments.
Hence the half-circle moments of $\chi_N$ agree with those of $\rho$,
and therefore
\[
\int_0^{2\pi}\chi_N(\theta)\cos\theta\,d\theta=0,
\]
\[
\int_0^{2\pi}\chi_N(\theta)\sin\theta\,d\theta=0.
\]

The measure \(\chi_N(\theta)\,d\theta\) is nonnegative, has total mass
\(\pi\), and is absolutely continuous, hence it is not concentrated on a
pair of antipodal points.  Together with the two closing conditions, the
planar Minkowski existence theorem in support-function
form~\cite{Schneider2014} gives a convex body, unique up to translation,
whose curvature-radius measure is \(\chi_N(\theta)\,d\theta\).  Thus the
equation
\[
h_N+h_N''=\chi_N
\]
has a $2\pi$-periodic support-function solution $h_N$, unique up to
addition of first harmonics, and this solution is the support function
of a convex body $R_N$.

The antipodal identity gives the width.  If
\[
w_N(\theta)=h_N(\theta)+h_N(\theta+\pi),
\]
then
\[
w_N+w_N''
=
\chi_N(\theta)+\chi_N(\theta+\pi)
=1.
\]
Since $w_N$ is $\pi$-periodic, the first-harmonic part in the general
solution of $w+w''=1$ vanishes, and hence $w_N\equiv1$.  Thus $R_N$ has
constant width $1$.

It remains to check that $R_N$ belongs to a finite Reuleaux level.
The selector $E_N$ may be chosen, up to null sets, as a union of $O(N)$
intervals; we use such a representative of $\chi_N$.  After merging
adjacent intervals and deleting zero-length pieces, let $n_N$ be the
number of connected components of the set $\{\chi_N=1\}$ on $S^1$.
The antipodal rule
\[
\chi_N(\theta)+\chi_N(\theta+\pi)=1,
\]
forces $n_N$ to be odd.  To see this, choose a point $\theta_0$ that is
not a switching point.  The values at $\theta_0$ and $\theta_0+\pi$ are
opposite, so the number of switches on the half-turn
$(\theta_0,\theta_0+\pi)$ is odd.  The switches on the opposite
half-turn are its translates by $\pi$, so the total number of switches
is twice this odd number, and the number of $1$-components is the same
odd number.
The one-component case is excluded by the closing conditions: if
\(\{\chi_N=1\}\) had a single connected component, the antipodal rule
would force it to be a semicircle up to null sets, and the first moment
of a semicircle is nonzero.  Hence \(n_N\ge3\).

On each component where $\chi_N=1$, the curvature radius is $1$, so the
corresponding boundary piece is a circular arc of radius $1$; on the
components where $\chi_N=0$, the curvature radius vanishes and the
boundary contribution is a vertex.  Therefore the boundary of $R_N$
consists of $n_N$ circular arcs of radius $1$, with $n_N$ odd and
$n_N=O(N)$.  Thus there exists a universal constant $A>0$ such that
\[
n_N\le AN
\]
for all $N$, and
\[
R_N\in\mathcal R_{n_N}.
\]
The complexity chain used here is
\[
N\ \text{mesh intervals}
\Longrightarrow
2N+2\ \text{moment constraints}
\Longrightarrow
O(N)\ \text{switches}
\Longrightarrow
n_N\le AN.
\]

It remains to estimate the Hausdorff error.  Let
\[
f_N=\chi_N-\rho,
\qquad
g_N=h_N-h_W.
\]
Then
\[
g_N+g_N''=f_N.
\]
We normalize $g_N$ to have no first-harmonic part; this only fixes the
translation of $R_N$.  Since both $\chi_N$ and $\rho$ satisfy the
constant-width antipodal identity, $f_N(\theta+\pi)=-f_N(\theta)$; hence
the same zeroth and first local moment identities hold on the translated
intervals $I_j+\pi$.  The local moment identities on
$I_j$ and $I_j+\pi$, together with the global first-harmonic identities,
put $f_N$ in the hypotheses of
Lemma~\ref{lem:green-estimate}.  Hence
\[
\|g_N\|_\infty\le C\Delta^2.
\]
Since Hausdorff distance between convex bodies in this normalized
placement equals the
$L^\infty$-distance between support functions after a common
translation normalization,
\[
d_{\mathcal E}(W,R_N)
\le
\|h_W-h_N\|_\infty
\le
C\Delta^2
=
\frac{C}{N^2}.
\]

Since \(R_N\in\mathcal R_{n_N}\) with \(n_N\le AN\), and since
\(\mathcal A_M\) contains all Reuleaux-type bodies with odd arc
complexity at most \(2M+1\), the body \(R_N\) belongs to
\(\mathcal A_M\) whenever \(AN\le 2M+1\).  For all sufficiently large
\(M\), choose
\[
N=\left\lfloor\frac{2M+1}{A}\right\rfloor\ge1.
\]
Then \(R_N\in\mathcal A_M\) and \(N\asymp M\), so
\[
\inf_{R\in\mathcal A_M}d_{\mathcal E}(W,R)
\le
\frac{C'}{M^2}.
\]
Increasing \(C'\), if necessary, handles the remaining finitely many
small values of \(M\).
Taking the supremum over $W\in\mathcal W_1$ gives
\[
\delta_M\le C_{\rm app}M^{-2}.
\]
\end{proof}

\subsection{The \texorpdfstring{$M^{-2}$}{M minus two} finite-arc approximation scale}
\label{sec:lower-scale}

The upper bound $\delta_M=O(M^{-2})$ is sufficient for the convergence
rate of the variational hierarchy.  Nevertheless, it is useful to
record that this rate is the natural Hausdorff approximation scale for
finite-arc Reuleaux-type bodies.

\begin{proposition}[Lower Hausdorff approximation rate]\label{prop:lower-rate}
There exists a constant $c_{\rm app}>0$ such that, for all
sufficiently large $M$,
\[
\delta_M\ge c_{\rm app}M^{-2}.
\]
Consequently,
\[
\delta_M=\Theta(M^{-2}).
\]
\end{proposition}

\begin{proof}
We construct a smooth constant-width body that cannot be approximated
faster than order $M^{-2}$ by Reuleaux-type bodies from the first $M$ hierarchy levels.

Choose $0<\varepsilon<1/4$, and define
\[
\rho_*(\theta)
=
\frac12+\varepsilon\cos 3\theta.
\]
Then
\[
\rho_*(\theta+\pi)
=
\frac12-\varepsilon\cos 3\theta,
\]
so
\[
\rho_*(\theta)+\rho_*(\theta+\pi)=1.
\]
Moreover,
\[
\frac14\le \rho_*(\theta)\le \frac34.
\]

The function $\rho_*$ has no first harmonics.  Hence there exists a
$2\pi$-periodic support function $h_*$, unique up to translation,
satisfying
\[
h_*+h_*''=\rho_*.
\]
Since $\rho_*\ge 1/4>0$, this support function has strictly positive
radius-of-curvature density, and therefore determines a strictly convex
body $K_*$.  The antipodal identity above implies that $K_*$ has
constant width $1$.  Its curvature-radius density is bounded strictly
away from $0$ and $1$:
there exists $\eta>0$ such that
\[
\eta\le \rho_*(\theta)\le 1-\eta
\]
for every $\theta$.

Let $P\in\mathcal R_{\le 2M+1}$.  Its curvature-radius density $\rho_P$
is $0/1$-valued, and the set \(\{\rho_P=1\}\) has at most \(2M+1\)
connected components.  Therefore the switching points partition \(S^1\)
into \(O(M)\) intervals on each of which \(\rho_P\) is constant.  Hence
at least one constant interval has length at least \(c/M\), with \(c>0\)
universal.  Choose a closed subinterval
\(I=[a,a+L]\subset S^1\) inside that constant interval, with
\[
L\ge \frac{c_0}{M}
\]
for another universal constant $c_0>0$.  On \(I\) one has
\(\rho_P\equiv0\) or \(\rho_P\equiv1\).  Since
$\rho_*\in[\eta,1-\eta]$, the function
\[
\rho_P-\rho_*
\]
has fixed sign on $I$, and
\[
|\rho_P(\theta)-\rho_*(\theta)|\ge \eta
\qquad
\text{for every }\theta\in I.
\]

Let
\[
u=h_P-h_*.
\]
Then
\[
u+u''=\rho_P-\rho_*.
\]

Choose a nonnegative function $\psi\in C_c^2(0,1)$,
$\psi\not\equiv0$, and define
\[
\psi_I(\theta)
=
\psi\left(\frac{\theta-a}{L}\right)
\]
for $\theta\in I$, with $\psi_I=0$ outside $I$.  Since $\psi_I$ and its
first derivative vanish at the endpoints of $I$, integration by parts
gives
\[
\int_I (u+u'')\psi_I\,d\theta
=
\int_I u(\psi_I+\psi_I'')\,d\theta.
\]

Because $\rho_P-\rho_*$ has fixed sign and absolute value at least
$\eta$ on $I$,
\[
\left|
\int_I (u+u'')\psi_I\,d\theta
\right|
\ge
\eta\int_I\psi_I\,d\theta.
\]

On the other hand,
\[
\left|
\int_I u(\psi_I+\psi_I'')\,d\theta
\right|
\le
\|u\|_\infty
\int_I|\psi_I+\psi_I''|\,d\theta.
\]

We have
\[
\int_I\psi_I\,d\theta
=
L\int_0^1\psi(s)\,ds
=
A_\psi L,
\]
where $A_\psi>0$, and
\[
\psi_I''(\theta)
=
L^{-2}\psi''\left(\frac{\theta-a}{L}\right).
\]
Therefore
\[
\int_I|\psi_I+\psi_I''|\,d\theta
\le
L\int_0^1|\psi(s)|\,ds
+
L^{-1}\int_0^1|\psi''(s)|\,ds
\le
B_\psi L^{-1}
\]
for all sufficiently small $L$, where $B_\psi>0$ is independent of $I$
and $M$.  Combining the preceding estimates gives
\[
\eta A_\psi L
\le
\|u\|_\infty B_\psi L^{-1}.
\]
Thus
\[
\|u\|_\infty
\ge
cL^2
\ge
c'M^{-2}.
\]

This estimate is invariant under translations, because adding a first
harmonic $q$ to $u$ leaves $u+u''$ unchanged and the
integration-by-parts identity above remains valid:
\[
\int_I q(\psi_I+\psi_I'')\,d\theta
=
\int_I(q+q'')\psi_I\,d\theta
=
0.
\]

It is also unaffected by rotations or reflections of $P$, since these
operations only shift or reverse the $0/1$-valued density $\rho_P$ and
preserve the existence of an interval $I$ of length $\gtrsim M^{-1}$ on
which $\rho_P$ is constant, since the number of arcs is at most
\(2M+1\).

Therefore every congruent copy $\widetilde P$ of every
$P\in\mathcal R_{\le 2M+1}$ satisfies
\[
d_H(\widetilde P,K_*)\ge c'M^{-2}.
\]
In the congruence-invariant distance,
\[
d_{\mathcal E}(P,K_*)\ge c'M^{-2}.
\]

The family $\mathcal A_M$ also contains the disk $D$.  The disk of
width \(1\) has curvature-radius density identically \(1/2\), whereas
\(\rho_*(\theta)=1/2+\varepsilon\cos3\theta\) is not constant.  Hence
\(K_*\ne D\), and so
\[
d_{\mathcal E}(K_*,D)=d_0>0.
\]
For all sufficiently large $M$, $d_0\ge c'M^{-2}$.  Hence
\[
\inf_{R\in\mathcal A_M}d_{\mathcal E}(K_*,R)
\ge
c_{\rm app}M^{-2}.
\]

Taking the supremum over $W\in\mathcal W_1$ gives
\[
\delta_M\ge c_{\rm app}M^{-2}.
\]
Together with Proposition~\ref{prop:upper-rate}, this proves
\[
\delta_M=\Theta(M^{-2}).
\]
\end{proof}

This proposition identifies the intrinsic Hausdorff approximation
scale of the Reuleaux hierarchy.  The next subsection transfers this
approximation scale to the variational covering values.

\subsection{Transfer from Hausdorff approximation to covering area}
\label{sec:transfer}

Hausdorff approximation of test bodies controls the error of the
variational lower-bound hierarchy.

\begin{lemma}[Area transfer lemma]\label{lem:area-transfer}
There exist constants $C_{\rm tr},C_{\rm St}>0$ such that, for every
$M$,
\[
\aLeb-\Lambda_M
\le
C_{\rm tr}\delta_M+C_{\rm St}\delta_M^2.
\]
\end{lemma}

\begin{proof}
Fix $\varepsilon>0$.  Choose a compact convex set $C_M$ that covers
every element of $\mathcal A_M$ and satisfies
\[
\operatorname{area}(C_M)\le \Lambda_M+\varepsilon.
\]

Since $D\in\mathcal A_M$, the set $C_M$ contains a congruent copy of
$D$.  Applying a global rigid motion to $C_M$, which preserves both
area and covering properties, we may assume that
\[
D\subset C_M,
\]
where $D$ is centered at the origin.

A uniform perimeter bound comes first.  Because $D\subset C_M$,
\[
h_{C_M}(u)\ge \frac12
\qquad
\text{for every }u\in S^1.
\]
Thus the fixed disk gives a uniform lower support bound in every
direction.
For planar convex bodies, the Minkowski support formula in
surface-area-measure form gives
\[
\operatorname{area}(C_M)
=
\frac12
\int_{S^1}h_{C_M}(u)\,dS_{C_M}(u),
\]
where \(S_{C_M}\) is the surface-area measure of \(C_M\).  Its total
mass is the perimeter:
\[
S_{C_M}(S^1)=\operatorname{per}(C_M).
\]
Therefore
\[
\operatorname{area}(C_M)
\ge
\frac12\cdot \frac12 \operatorname{per}(C_M)
=
\frac14\operatorname{per}(C_M).
\]
Hence
\[
\operatorname{per}(C_M)\le 4\operatorname{area}(C_M).
\]

Let $A_*$ be any fixed finite upper bound for $\aLeb$, obtained for
instance from any known universal cover.  Since
\[
\Lambda_M\le \aLeb\le A_*,
\]
we have, for $0<\varepsilon\le1$,
\[
\operatorname{area}(C_M)\le A_*+1.
\]
Thus
\[
\operatorname{per}(C_M)\le 4(A_*+1)=:P_*.
\]

Now let $W\in\mathcal W_1$ be arbitrary, and fix $\alpha>0$.  By the
definition of $\delta_M$, there exist $R_W\in\mathcal A_M$ and an
isometry $q_W\in\mathcal E(2)$ such that
\[
d_H(W,q_WR_W)\le\delta_M+\alpha.
\]

Since $C_M$ covers $R_W$, there exists a rigid motion
$g_W\in\mathcal E(2)$ such that
\[
g_WR_W\subset C_M.
\]
Set
\[
\widetilde g_W=g_Wq_W^{-1}.
\]
Hausdorff distance is invariant under rigid motions, so
\[
d_H(\widetilde g_WW,g_WR_W)\le\delta_M+\alpha.
\]
Therefore
\[
\widetilde g_WW
\subset
g_WR_W+(\delta_M+\alpha) B
\subset
C_M+(\delta_M+\alpha) B,
\]
where $B$ is the unit disk centered at the origin.  Since
$W\in\mathcal W_1$ was arbitrary,
\[
C_M+(\delta_M+\alpha)B
\]
contains a congruent copy of every constant-width body of width $1$.
By the constant-width reduction proved in Section~\ref{sec:constant-width},
$C_M+(\delta_M+\alpha)B$ is a universal cover for planar sets of diameter at
most one.  Hence
\[
\aLeb
\le
\operatorname{area}(C_M+(\delta_M+\alpha) B).
\]

By Steiner's formula for planar convex bodies,
\[
\operatorname{area}(C_M+tB)
=
\operatorname{area}(C_M)
+
t\,\operatorname{per}(C_M)
+
\pi t^2.
\]
Taking $t=\delta_M+\alpha$, we obtain
\[
\aLeb
\le
\operatorname{area}(C_M)
+
(\delta_M+\alpha)\operatorname{per}(C_M)
+
\pi(\delta_M+\alpha)^2.
\]

Using
\[
\operatorname{area}(C_M)\le \Lambda_M+\varepsilon
\]
and
\[
\operatorname{per}(C_M)\le P_*,
\]
we get
\[
\aLeb
\le
\Lambda_M+\varepsilon
+
P_*(\delta_M+\alpha)
+
\pi(\delta_M+\alpha)^2.
\]
Letting $\varepsilon\downarrow0$ and then $\alpha\downarrow0$, we obtain
\[
\aLeb-\Lambda_M
\le
P_*\delta_M+\pi\delta_M^2.
\]
Thus the desired estimate holds with
\[
C_{\rm tr}=P_*=4(A_*+1),
\qquad
C_{\rm St}=\pi.
\]
Both constants are fixed once the auxiliary universal-cover upper bound
\(A_*\) is fixed; in particular they are independent of \(M\).
\end{proof}

\subsection{Main quantitative convergence theorem}
\label{sec:main-quant}

Combining the Hausdorff approximation estimate with the area transfer
lemma gives the quantitative convergence theorem.

\begin{theorem}[Quantitative convergence of the variational hierarchy]\label{thm:quantitative}
There exists a constant $C>0$ such that
\[
0\le \aLeb-\Lambda_M\le C M^{-2}
\]
for all $M\ge1$.
\end{theorem}

\begin{proof}
The lower inequality
\[
0\le \aLeb-\Lambda_M
\]
is precisely the lower-bound property of the hierarchy, proved in
Section~\ref{sec:variational-hierarchy}:
\[
\Lambda_M\le \aLeb.
\]

For the upper estimate, Lemma~\ref{lem:area-transfer} gives
\[
\aLeb-\Lambda_M
\le
C_{\rm tr}\delta_M+C_{\rm St}\delta_M^2.
\]

By Proposition~\ref{prop:upper-rate},
\[
\delta_M\le C_{\rm app}M^{-2}.
\]
Therefore
\[
\aLeb-\Lambda_M
\le
C_{\rm tr}C_{\rm app}M^{-2}
+
C_{\rm St}C_{\rm app}^2M^{-4}.
\]

Since $M\ge1$, the $M^{-4}$ term is bounded by a constant multiple of
$M^{-2}$.  Thus one may take, for example,
\[
C=C_{\rm tr}C_{\rm app}+C_{\rm St}C_{\rm app}^2.
\]
With this choice,
\[
\aLeb-\Lambda_M
\le
C M^{-2}.
\]

Combining the two inequalities gives
\[
0\le \aLeb-\Lambda_M\le C M^{-2}.
\]
\end{proof}

\begin{corollary}[Tail stability of the hierarchy]\label{cor:tail-stability}
For every \(N\ge M\), the hierarchy satisfies
\[
0\le \Lambda_N-\Lambda_M\le C M^{-2}.
\]
In particular, successive hierarchy levels satisfy
\[
0\le \Lambda_{M+1}-\Lambda_M\le C M^{-2}.
\]
\end{corollary}

\begin{proof}
The lower bound follows from monotonicity.  Since \(\Lambda_N\le\aLeb\),
we have
\[
\Lambda_N-\Lambda_M
\le
\aLeb-\Lambda_M
\le
C M^{-2},
\]
where the last inequality is Theorem~\ref{thm:quantitative}.
The estimate for \(N=M+1\) is the special case of the same inequality.
\end{proof}

Theorem~\ref{thm:quantitative} is the quantitative strengthening of
the exactness theorem.  It shows that the Reuleaux-type variational
hierarchy converges to $\aLeb$ at the \(M^{-2}\) finite-arc approximation
rate.  Corollary~\ref{cor:tail-stability} also shows that the tail of the
hierarchy is stable at the same scale.

We record the precise relationship among the estimates proved in this
section:
\[
\delta_M=\Theta(M^{-2})
\]
is a statement about the optimal Hausdorff approximation rate of the
Reuleaux-type families;
\[
\aLeb-\Lambda_M
\le
C_{\rm tr}\delta_M+C_{\rm St}\delta_M^2
\]
is the transfer estimate from Hausdorff approximation of test bodies
to covering-area convergence;
and
\[
0\le \aLeb-\Lambda_M\le C M^{-2}
\]
is the quantitative convergence rate of the variational lower-bound
hierarchy.

The matching lower bound in \(\delta_M=\Theta(M^{-2})\) is a Hausdorff
approximation statement.  The hierarchy value uses the transfer estimate
above, which gives the one-sided convergence rate
\[
0\le \aLeb-\Lambda_M\le C M^{-2}.
\]


\section{Organization of the Certified Finite-Test Proof}
\label{sec:roadmap}

Having established the exact Reuleaux-type hierarchy, we now record the
finite-test proof used for the certified lower bound.  This section
recaps the finite-test setup described in the main text and fixes the
notation used in Sections~\ref{sec:low-order}--\ref{sec:interval-certification}.
The finite test is extracted from the second hierarchy level, reduced to
a five-parameter placement problem, and certified by interval arithmetic,
yielding
\[
\aLeb\ge 0.834.
\]

The hierarchy in Part~I is built from
\[
\mathcal A_M
=
\{D\}\cup\bigcup_{m=1}^{M}\mathcal R_{2m+1}.
\]

Only odd arc counts enter the hierarchy.  The \(M\)-th hierarchy
level contains the Reuleaux-type families with arc counts
\[
3,5,\ldots,2M+1.
\]

The second hierarchy level is the low-order truncation involving the
disk and the two smallest nontrivial odd Reuleaux families:
\[
\mathcal A_2
=
\{D\}\cup\mathcal R_3\cup\mathcal R_5.
\]
This truncation is geometrically natural.  The disk $D$ is the most
symmetric constant-width body of width $1$, while $\mathcal R_3$ and
$\mathcal R_5$ are the first finite-arc Reuleaux-type families in the
hierarchy.

From this level we extract the regular representatives
\[
R_3\in\mathcal R_3,
\qquad
R_5\in\mathcal R_5,
\]
where $R_3$ is the regular Reuleaux triangle of width $1$, and $R_5$
is the regular Reuleaux pentagon of width $1$.  Together with the disk
$D$, they form the finite test family
\[
\mathcal F_{3,5}^{\rm R}
:=
\{D,R_3,R_5\}
\subset\mathcal W_1.
\]

Define
\[
L_{3,5}^{\rm R}
:=
\Lambda(\mathcal F_{3,5}^{\rm R}).
\]
Since \(\mathcal F_{3,5}^{\rm R}\subset\mathcal W_1\), the admissible
class for covering all width-one constant-width bodies is contained in
the admissible class for covering \(\mathcal F_{3,5}^{\rm R}\), so
\[
L_{3,5}^{\rm R}\le \aLeb.
\]
It remains only to prove
\[
L_{3,5}^{\rm R}\ge 0.834
\]
because this implies
\[
\aLeb\ge 0.834.
\]

The value $L_{3,5}^{\rm R}$ can be written as an optimization over relative
placements of $D,R_3,R_5$.  Since the area of the convex hull is
invariant under global Euclidean motions, we fix the disk $D$ at the
origin:
\[
D=\{x\in\R^2:|x|\le 1/2\}.
\]
The remaining degrees of freedom are the relative placements of $R_3$
and $R_5$.  After using global motion and symmetry normalizations,
these placements are encoded by a five-dimensional parameter vector
\[
z=(x_3,y_3,\psi_5,x_5,y_5).
\]
Here $(x_3,y_3)$ denotes the translation of the regular Reuleaux
triangle in the chosen normalization, $\psi_5$ denotes the relative
rotation of the regular Reuleaux pentagon, and $(x_5,y_5)$ denotes the
translation of the pentagon.  The global orthogonal freedom, together
with the symmetry of \(R_3\), is used to fix the orientation of the
triangle.

The corresponding area function is
\[
A(z)
=
\operatorname{area}
\operatorname{conv}
\bigl(
D\cup R_3(z)\cup R_5(z)
\bigr).
\]
With the normalization proved in Section~\ref{sec:normalization},
\[
L_{3,5}^{\rm R}
=
\inf_z A(z),
\]
where the infimum is taken over all admissible normalized
configurations.  The certified lower-bound problem is
\[
\inf_z A(z)\ge 0.834.
\]
The rest of Sections~\ref{sec:low-order}--\ref{sec:interval-certification}
proves this finite-dimensional inequality.

The normalized parameter space is still unbounded in the translation
variables.  We reduce it to a compact region by a sequence of analytic
estimates and interval certificates.

Elementary disk-plus-point estimates show that if the center of
$R_3$ or $R_5$ is sufficiently far from the origin, then the convex
hull already has area greater than $0.834$, giving a coarse bounded
domain.

One-body radial collar certificates remove the remaining cases
in which either $R_3$ or $R_5$ lies outside a smaller core box.  After
these reductions, all configurations not already certified to have area
at least \(0.834\) are contained in the compact five-dimensional core
domain
\[
\Omega_{\rm core}
=
[-0.08,0.08]^2
\times
[0,\pi/5]
\times
[-0.08,0.08]^2.
\]
The angle interval $[0,\pi/5]$ comes from the rotational and
reflection symmetries of the regular Reuleaux pentagon.

The final core task is therefore to prove
\[
A(z)\ge 0.834
\qquad
\text{for every }z\in\Omega_{\rm core}.
\]
Together with the collar certificates, this is carried out by
subdividing the remaining compact certified domains into boxes and
proving a certified lower bound on each box using interval arithmetic.

The certificate has the following form.  For each remaining compact
parameter domain \(\Omega\), the interval verifier checks a finite box partition
\[
\Omega
=
\bigcup_\alpha B_\alpha.
\]
These domains include the core domain above and the one-body collar
domains left by the preceding reductions.

For every leaf box \(B_\alpha\), the certificate verifies the appropriate
interval lower-bound statement for that domain; on core boxes this is
\[
\inf_{z\in B_\alpha}A(z)\ge 0.834.
\]
The verification uses interval arithmetic and local convex-hull lower
bounds.  Different boxes may be certified by different local methods,
such as inner-polygon bounds, fixed-template convexity bounds, or
simple-template mean-value bounds.  Together, these local methods provide
a rigorous finite proof that no configuration in the certified parameter
domains has area below \(0.834\).

A deterministic interval verification then checks the entire certificate:
it verifies the constants, the domain reductions, the box coverage, the
absence of gaps between certified regions, and the local lower bound on
every leaf box.  Once all checks pass, the certificate proves in
particular
\[
\inf_{z\in\Omega_{\rm core}} A(z)\ge 0.834.
\]
Together with the collar certificates and the analytic reductions, this
proves
\[
L_{3,5}^{\rm R}\ge 0.834.
\]
Since
\[
L_{3,5}^{\rm R}\le \aLeb,
\]
the final conclusion is
\[
\aLeb\ge 0.834.
\]

Subsequent sections supply the finite-test comparison, the normalization,
the analytic and collar reductions, the local interval methods, and the
final certified lower-bound theorem.

\section{The finite Reuleaux test family \texorpdfstring{$\mathcal F_{3,5}^{\rm R}$}{F3,5R}}
\label{sec:low-order}

The finite Reuleaux subtest used in the certificate consists of the disk
and two regular representatives from the second hierarchy level:
\[
D,\qquad R_3,\qquad R_5,
\]
where \(D\) is the disk of radius \(1/2\) and
\[
R_3:=R_3^{\rm reg}\in\mathcal R_3,
\qquad
R_5:=R_5^{\rm reg}\in\mathcal R_5
\]
are the regular Reuleaux triangle and regular Reuleaux pentagon of
width \(1\).
Since all three bodies belong to $\mathcal W_1$, the
family
\[
\mathcal F_{3,5}^{\rm R}
:=
\{D,R_3,R_5\}
\]
is a legitimate finite test family for the Lebesgue universal covering
problem.  The regular subtest, the associated five-parameter placement
model, and the certification funnel are illustrated in Fig.~2 of the
main text.  The placement model can also be visualized in the
interactive demo~\cite{PlacementDemo}.

The goal of Part~II is to prove that this finite test family already
forces area at least $0.834$.

\subsection{The finite-test lower-bound value}

Recall that for any subfamily $\mathcal F\subseteq \mathcal W_1$, the
associated variational lower-bound value is
\[
\Lambda(\mathcal F)
=
\inf\left\{
\operatorname{area}(U):
\begin{array}{l}
U\subset\R^2 \text{ is compact and convex, and }U\\
\text{contains a congruent copy of every }K\in\mathcal F
\end{array}
\right\}.
\]

For the finite regular Reuleaux family $\mathcal F_{3,5}^{\rm R}$, define
\[
L_{3,5}^{\rm R}
:=
\Lambda(\mathcal F_{3,5}^{\rm R}).
\]
Since
\[
\mathcal F_{3,5}^{\rm R}\subset\mathcal W_1,
\]
the admissible class for covering all width-one constant-width bodies is
contained in the admissible class for covering
\(\mathcal F_{3,5}^{\rm R}\), so
\[
L_{3,5}^{\rm R}\le \aLeb.
\]
Any rigorous lower bound on $L_{3,5}^{\rm R}$ is automatically a
rigorous lower bound on $\aLeb$.  In particular,
\[
L_{3,5}^{\rm R}\ge 0.834
\qquad\Longrightarrow\qquad
\aLeb\ge 0.834.
\]
This implication is the fundamental reason why the finite Reuleaux test
family $\mathcal F_{3,5}^{\rm R}$ is sufficient for the certified lower bound.

\subsection{Convex-hull formulation}

Because $\mathcal F_{3,5}^{\rm R}$ is finite, $L_{3,5}^{\rm R}$ can be written as an
infimum over relative placements.  Let $g_0,g_3,g_5\in\mathcal E(2)$.
A convex set containing congruent copies of $D,R_3,R_5$ at these
placements must contain
\[
\operatorname{conv}
\bigl(
g_0D\cup g_3R_3\cup g_5R_5
\bigr).
\]
Conversely, this convex hull itself is a compact convex set containing
the three placed bodies.  Hence
\[
L_{3,5}^{\rm R}
=
\inf_{g_0,g_3,g_5\in\mathcal E(2)}
\operatorname{area}
\operatorname{conv}
\bigl(
g_0D\cup g_3R_3\cup g_5R_5
\bigr).
\]

The area is invariant under a common Euclidean motion.  Applying
\(g_0^{-1}\) to all three placed bodies does not change the convex-hull
area, so we may fix
\[
g_0D=D,
\qquad
D=\{x\in\R^2:|x|\le 1/2\},
\]
and optimize only over the relative placements of $R_3$ and $R_5$.
Thus
\[
L_{3,5}^{\rm R}
=
\inf_{g_3,g_5\in\mathcal E(2)}
\operatorname{area}
\operatorname{conv}
\bigl(
D\cup g_3R_3\cup g_5R_5
\bigr).
\]
This expression is the geometric core of the finite test problem.

\subsection{The area functional}

In the complete normalization proved in Section~\ref{sec:normalization},
the remaining relative
configuration is represented by a five-dimensional parameter vector
\[
z=(x_3,y_3,\psi_5,x_5,y_5).
\]
The exact normalization and symmetry reductions are described in
Section~\ref{sec:normalization}.  For the present section, it is
enough to view $z$ as a complete parameter for the relative placements
of $R_3$ and $R_5$ after fixing the disk $D$.

Let $R_3(z)$ denote the placed regular Reuleaux triangle, and let
$R_5(z)$ denote the placed regular Reuleaux pentagon.  Define
\[
A(z)
=
\operatorname{area}
\operatorname{conv}
\bigl(
D\cup R_3(z)\cup R_5(z)
\bigr).
\]
With this normalization,
\[
L_{3,5}^{\rm R}
=
\inf_z A(z),
\]
where the infimum is over all admissible normalized configurations.
The certified lower-bound problem is therefore reduced to proving
\[
A(z)\ge 0.834
\]
for every admissible parameter value $z$.

Sections~\ref{sec:normalization}--\ref{sec:interval-certification} establish this
inequality by analytic reductions and interval certificates on the
remaining compact parameter domains.

\subsection{Relation to the polygonal test family \texorpdfstring{$\mathcal F_{3,5}^{\rm P}$}{F3,5P}}

The finite Reuleaux test is naturally related to the polygonal
finite-test method used by Brass and Sharifi~\cite{BrassSharifi2005},
which is based on the disk, an equilateral triangle, and a regular
pentagon.  We show below that, at the level of finite-test values, the
Reuleaux test is a mathematical strengthening of this polygonal test:
its finite-test value is at least as large.  This comparison is relevant
because the Brass--Sharifi polygonal finite test gives the previously
published lower-bound benchmark
\[
\aLeb\ge0.832.
\]
Let
$\triangle$ be the equilateral triangle of diameter $1$, and let $P_5$
be the regular pentagon of diameter $1$, placed so that
\[
\triangle\subset R_3,
\qquad
P_5\subset R_5.
\]
Define the polygonal finite test family
\[
\mathcal F_{3,5}^{\rm P}
:=
\{D,\triangle,P_5\},
\qquad
L_{3,5}^{\rm P}
:=
\Lambda(\mathcal F_{3,5}^{\rm P}).
\]

For any common normalized placement $z$, let
\[
A_{3,5}^{\rm P}(z)
=
\operatorname{area}
\operatorname{conv}
\bigl(
D\cup \triangle(z)\cup P_5(z)
\bigr),
\]
and
\[
A_{3,5}^{\rm R}(z)
=
\operatorname{area}
\operatorname{conv}
\bigl(
D\cup R_3(z)\cup R_5(z)
\bigr)
=
A(z).
\]

Since
\[
\triangle(z)\subset R_3(z),
\qquad
P_5(z)\subset R_5(z),
\]
we have the pointwise inclusion
\[
\operatorname{conv}
\bigl(
D\cup \triangle(z)\cup P_5(z)
\bigr)
\subset
\operatorname{conv}
\bigl(
D\cup R_3(z)\cup R_5(z)
\bigr).
\]
Therefore
\[
A_{3,5}^{\rm P}(z)\le A_{3,5}^{\rm R}(z)
\]
for every normalized configuration \(z\).  The common normalized
placement model gives only a restricted class of polygonal placements,
whereas \(L_{3,5}^{\rm P}\) is the full polygonal finite-test infimum.
Hence
\[
L_{3,5}^{\rm P}
\le
\inf_z A_{3,5}^{\rm P}(z)
\le
\inf_z A_{3,5}^{\rm R}(z)
=
L_{3,5}^{\rm R},
\]
where the last equality is the complete Reuleaux normalization proved in
Section~\ref{sec:normalization}.

Thus the finite Reuleaux test family \(\mathcal F_{3,5}^{\rm R}\)
is no weaker, at the level of finite-test values, than the polygonal
test family \(\mathcal F_{3,5}^{\rm P}\) associated with the
Brass--Sharifi finite-test method:
\[
L_{3,5}^{\rm P}\le L_{3,5}^{\rm R}.
\]
This is a purely geometric comparison, independent of the numerical
certificate: replacing the polygonal test bodies by their Reuleaux
constant-width completions can only increase the convex-hull area for
each fixed placement, and hence cannot weaken the resulting finite-test
lower bound.

\subsection{What remains to be certified}

The preceding subsections reduce the certified lower-bound theorem to
the following finite-dimensional statement:
\[
\inf_z
\operatorname{area}
\operatorname{conv}
\bigl(
D\cup R_3(z)\cup R_5(z)
\bigr)
\ge 0.834.
\]
The rest of Part~II proves this statement by a five-parameter
normalization, analytic and collar reductions to a compact core box,
local interval lower bounds, and a final certificate assembly, concluding
\[
\aLeb\ge 0.834.
\]

\section{Complete Five-Parameter Normalization}
\label{sec:normalization}

The parameter normalization below applies to the finite test problem
\[
L_{3,5}^{\rm R}
=
\inf
\operatorname{area}
\operatorname{conv}
\bigl(
g_0D\cup g_3R_3\cup g_5R_5
\bigr),
\]
where the infimum is taken over all Euclidean placements of
$D,R_3,R_5$.  The goal of this section is to show that, after using
the global Euclidean invariance and the intrinsic symmetries of the
regular Reuleaux bodies, every configuration can be represented by a
five-dimensional parameter vector
\[
z=(x_3,y_3,\psi_5,x_5,y_5).
\]
The bounded reduction of these five parameters is treated later in
Sections~\ref{sec:analytic-reduction} and~\ref{sec:collar}.

\subsection{Reference bodies}

Let
\[
D=\{x\in\R^2:|x|\le 1/2\}
\]
be the disk of radius $1/2$ centered at the origin.

Fix once and for all reference copies $R_3^0$ and $R_5^0$ of the
regular Reuleaux triangle and regular Reuleaux pentagon of width $1$.
We choose $R_3^0$ and $R_5^0$ to be centered at the origin.  Their
precise initial orientation is irrelevant, but it will be kept fixed
throughout the certificate.  Here ``centered'' means that the common
dihedral symmetry center of the regular Reuleaux body, equivalently the
circumcenter of the underlying regular polygon, is placed at the origin.

For $\psi\in\R$, let $Q_\psi$ be the counterclockwise rotation matrix
by angle $\psi$.  For $p\in\R^2$, write
\[
T_p(x)=x+p
\]
for translation by $p$.

The regular Reuleaux triangle and pentagon have dihedral symmetry,
including reflection symmetry.  Hence an orientation-reversing
congruent copy is also representable, after using a symmetry of the
reference body, by a rotation followed by a translation.  Thus placed
regular Reuleaux triangle and pentagon copies have the form
\[
T_pQ_\alpha R_3^0,
\qquad
T_qQ_\beta R_5^0.
\]

Since $D$ is invariant under orthogonal transformations about its
center, and the objective is invariant under a common Euclidean motion,
one global translation and one global orthogonal freedom may be used to
normalize the configuration.

\subsection{Fixing the disk and the orientation of \texorpdfstring{$R_3$}{R3}}

Consider an arbitrary configuration
\[
g_0D,\qquad g_3R_3^0,\qquad g_5R_5^0.
\]

Applying the inverse of $g_0$ to the whole configuration, we may
assume
\[
g_0D=D.
\]

Since the disk is invariant under orthogonal transformations about the
origin, there remains a global orthogonal freedom.  Using this freedom,
and the dihedral symmetry of the regular Reuleaux triangle, we fix the
orientation of \(R_3\).  In other words, after a further orthogonal
transformation about the origin, every configuration can be written in
the form
\[
D,\qquad
T_{p_3}R_3^0,\qquad
T_{p_5}Q_{\psi_5}R_5^0,
\]
where
\[
p_3=(x_3,y_3),
\qquad
p_5=(x_5,y_5),
\]
and $\psi_5$ is the relative angle of the pentagon with respect to the
fixed reference orientation of the triangle.

After these normalizations, the continuous degrees of freedom are
\[
x_3,\ y_3,\ \psi_5,\ x_5,\ y_5.
\]
The five-parameter representation is
\[
z=(x_3,y_3,\psi_5,x_5,y_5).
\]

For such a parameter vector, define
\[
R_3(z):=T_{(x_3,y_3)}R_3^0,
\]
and
\[
R_5(z):=T_{(x_5,y_5)}Q_{\psi_5}R_5^0.
\]

The corresponding area functional is
\[
A(z)
=
\operatorname{area}
\operatorname{conv}
\bigl(
D\cup R_3(z)\cup R_5(z)
\bigr).
\]

\subsection{Angle reduction for \texorpdfstring{$R_5$}{R5}}

The regular Reuleaux pentagon has fivefold rotational symmetry and
reflection symmetry, so the value of $A(z)$ is unchanged under
replacing
\[
\psi_5
\quad\text{by}\quad
\psi_5+\frac{2\pi k}{5},
\qquad k\in\Z,
\]
provided the same placed pentagon is represented in the equivalent
orientation.

Moreover, choose a reflection symmetry axis of the fixed reference
triangle \(R_3^0\).  Reflecting the entire normalized configuration
across this axis preserves both the disk and \(R_3^0\), while replacing
the relative pentagon angle \(\psi_5\) by \(-\psi_5\) and reflecting the
translation variables.  Since the translation variables range over all
of \(\R^2\), this reflection does not change the infimum of the area
functional.

For the purpose of minimizing $A(z)$, it is sufficient
to restrict the pentagon angle to the fundamental interval
\[
0\le \psi_5\le \frac{\pi}{5}.
\]

The normalized full parameter domain is
\[
\Omega_{\rm full}
=
\R^2
\times
\left[0,\frac{\pi}{5}\right]
\times
\R^2.
\]

Therefore
\[
L_{3,5}^{\rm R}
=
\inf_{z\in\Omega_{\rm full}} A(z).
\]

We emphasize that this is still an unbounded five-dimensional problem
because the translation variables $(x_3,y_3)$ and $(x_5,y_5)$ are not
yet restricted.  The reduction to a compact domain is carried out in
the next two sections.

\subsection{Completeness of the normalization}

We record the preceding discussion as a proposition.

\begin{proposition}\label{prop:normalization}
Every placement of $D,R_3,R_5$ is congruent to a normalized placement
of the form
\[
D,\qquad
T_{(x_3,y_3)}R_3^0,\qquad
T_{(x_5,y_5)}Q_{\psi_5}R_5^0,
\]
with
\[
(x_3,y_3,x_5,y_5)\in\R^4,
\qquad
\psi_5\in\left[0,\frac{\pi}{5}\right].
\]
Consequently,
\[
L_{3,5}^{\rm R}
=
\inf_{(x_3,y_3,\psi_5,x_5,y_5)\in\Omega_{\rm full}}
\operatorname{area}
\operatorname{conv}
\bigl(
D\cup T_{(x_3,y_3)}R_3^0
\cup
T_{(x_5,y_5)}Q_{\psi_5}R_5^0
\bigr).
\]
\end{proposition}

\begin{proof}
Start with an arbitrary placement of the three bodies.  A common
Euclidean motion sends the placed disk to the reference disk $D$
centered at the origin.  Since \(D\) is invariant under orthogonal
transformations about the origin, a further orthogonal transformation
can be used to align the placed Reuleaux triangle with the reference
orientation of \(R_3^0\).  If this alignment is orientation-reversing,
the reflection is absorbed by the dihedral symmetry of \(R_3^0\).  The
triangle is then represented only by its translation vector
\((x_3,y_3)\).

After this normalization, the pentagon is represented by its
translation vector \((x_5,y_5)\) and a relative rotation angle
\(\psi_5\).  The fivefold rotational symmetry of \(R_5^0\) reduces
\(\psi_5\) modulo \(2\pi/5\).  Reflection across a symmetry axis of
\(R_3^0\) preserves \(D\) and \(R_3^0\), identifies \(\psi_5\) with
\(-\psi_5\), and only reflects the translation variables.  Therefore
every configuration has a representative with
\[
0\le \psi_5\le \frac{\pi}{5}.
\]

Since the area of the convex hull is invariant under the common
Euclidean motions and reflections used in this normalization, taking
the infimum over all original placements is the same as taking the
infimum over $\Omega_{\rm full}$.
\end{proof}

The rest of Part~II works with the five-dimensional parameter vector
\[
z=(x_3,y_3,\psi_5,x_5,y_5)
\]
and the area functional
\[
A(z)
=
\operatorname{area}
\operatorname{conv}
\bigl(D\cup R_3(z)\cup R_5(z)\bigr).
\]
Analytic estimates reduce the unbounded domain $\Omega_{\rm full}$ to a
bounded parameter region.

\section{Analytic Reduction to a Bounded Parameter Domain}
\label{sec:analytic-reduction}

The five-parameter normalization in Section~\ref{sec:normalization}
reduces the finite test problem to
\[
L_{3,5}^{\rm R}
=
\inf_{z\in\Omega_{\rm full}} A(z),
\qquad
\Omega_{\rm full}
=
\R^2\times[0,\pi/5]\times\R^2,
\]
where
\[
z=(x_3,y_3,\psi_5,x_5,y_5)
\]
and
\[
A(z)
=
\operatorname{area}
\operatorname{conv}
\bigl(
D\cup R_3(z)\cup R_5(z)
\bigr).
\]

This domain is still unbounded in the four translation variables.  The
purpose of this section is to prove an analytic far-field reduction:
if either $R_3$ or $R_5$ is translated too far from the disk, then
the convex hull already has area greater than $0.834$.  Hence any
normalized configuration with
\[
A(z)<0.834
\]
must lie in a bounded region.

The estimates in this section use only elementary geometry: the convex
hull of the disk and a single far point already provides a lower bound
for the area of the full convex hull.

\subsection{The disk-plus-point area function}

Let $D=B(0,1/2)$, and let $p\in\R^2$ satisfy
\[
|p|=\rho\ge \frac12.
\]
Define
\[
G(\rho)
=
\operatorname{area}
\operatorname{conv}
\bigl(D\cup\{p\}\bigr).
\]

The following formula will be used repeatedly.

\begin{lemma}\label{lem:disk-point}
For $\rho\ge1/2$,
\[
G(\rho)
=
\frac{\pi}{4}
+
\frac{1}{2}\sqrt{\rho^2-\frac14}
-
\frac14
\arccos\left(\frac{1}{2\rho}\right).
\]
Moreover, $G$ is strictly increasing for $\rho>1/2$.
\end{lemma}

\begin{proof}
For $\rho=1/2$, the point $p$ lies on the disk and
$G(1/2)=\operatorname{area}(D)=\pi/4$.

Assume $\rho>1/2$.  Its area can be computed from the area of the disk
and the two tangent triangles from $p$, with the circular sector between
the tangent points subtracted once.

Let $r=1/2$, and let \(t_+\) and \(t_-\) be the two tangent points on
\(\partial D\).  Each triangle with vertices \(0,p,t_\pm\) is a right
triangle with legs \(r\) and \(\sqrt{\rho^2-r^2}\), so the two triangles
together contribute
\[
2\cdot\frac12 r\sqrt{\rho^2-r^2}
=
r\sqrt{\rho^2-r^2}.
\]
This adds the two tangent triangles to the disk, but the circular sector
of the disk between \(t_+\) and \(t_-\) has then been counted twice: once
as part of the disk and once inside the two tangent triangles.  The angle
of this sector is
\(2\arccos(r/\rho)\), and its area is
\[
\frac12 r^2\cdot 2\arccos(r/\rho)
=
r^2\arccos(r/\rho).
\]
Thus
\[
G(\rho)
=
\pi r^2
+
r\sqrt{\rho^2-r^2}
-
r^2\arccos(r/\rho).
\]

Substituting $r=1/2$ gives the stated formula.  Differentiating gives
\[
G'(\rho)
=
\frac{\sqrt{4\rho^2-1}}{4\rho}
=
\frac{\sqrt{\rho^2-\frac14}}{2\rho}
>0
\qquad
(\rho>1/2).
\]
Therefore $G$ is strictly increasing on $(1/2,\infty)$.
\end{proof}

The significance of $G$ is the following.  If a normalized configuration
contains a point $p$ with $|p|=\rho$, then
\[
\operatorname{conv}(D\cup\{p\})
\subset
\operatorname{conv}
\bigl(D\cup R_3(z)\cup R_5(z)\bigr),
\]
and hence
\[
A(z)\ge G(\rho).
\]
Therefore, once $G(\rho)>0.834$, any normalized configuration
containing such a point is automatically safe.

\subsection{Far-field reduction for \texorpdfstring{$R_3$}{R3}}

Let $c_3=(x_3,y_3)$ be the center of the placed regular Reuleaux
triangle $R_3(z)$, and write
\[
r_3=\|c_3\|.
\]

The three vertices of the underlying equilateral triangle of $R_3$ lie
at distance
\[
R_3^{\rm vert}=\frac{1}{\sqrt3}
\]
from the center.  For any direction in the plane, one of the three
vertex directions makes angle at most $\pi/3$ with it.  Hence, if
$r_3>0$, there is a vertex $v$ of $R_3(z)$ whose distance from the
origin is at least
\[
d_3(r_3)
=
\sqrt{
r_3^2
+
\left(R_3^{\rm vert}\right)^2
+
2r_3R_3^{\rm vert}\cos(\pi/3)
}.
\]

Since $R_3^{\rm vert}=1/\sqrt3$, this becomes
\[
d_3(r)
=
\sqrt{
r^2+\frac13+\frac{r}{\sqrt3}
}.
\]

The function $d_3(r)$ is strictly increasing for $r\ge0$; indeed,
differentiating
\[
d(r)=\sqrt{r^2+R^2+2rR\cos\delta}
\]
gives
\[
d'(r)=\frac{r+R\cos\delta}{d(r)}>0
\]
in the present case \(R=R_3^{\rm vert}\) and \(\delta=\pi/3\).

We use the cutoff
\[
b_3=0.190700.
\]

All computations in this far-field step are performed by
outward-rounded interval arithmetic at 200-bit precision.  The decimal
cutoffs are interpreted as exact decimal interval inputs.  For
\(b_3=0.190700\), the outward-rounded interval evaluation gives
\[
G(d_3(b_3))-0.834
\in
[2.13944145452936,\,2.13944145452937]\times 10^{-7},
\]
which is strictly positive.

Thus, by monotonicity of both $d_3$ and $G$, if
\[
\|c_3\|\ge b_3,
\]
then $R_3(z)$ contains a vertex $p$ such that
\[
G(|p|)\ge G(d_3(b_3))>0.834.
\]
Consequently,
\[
A(z)>0.834.
\]

We have proved the following far-field exclusion.

\begin{lemma}\label{lem:far-r3}
If a normalized configuration satisfies
\[
A(z)<0.834,
\]
then
\[
\|c_3\|<b_3=0.190700.
\]
\end{lemma}

\subsection{Far-field reduction for \texorpdfstring{$R_5$}{R5}}

Let $c_5=(x_5,y_5)$ be the center of the placed regular Reuleaux
pentagon $R_5(z)$, and write
\[
r_5=\|c_5\|.
\]

The five vertices of the underlying regular pentagon of $R_5$ lie at
distance
\[
R_5^{\rm vert}
=
\frac{1}{2\cos(\pi/10)}
\]
from the center; this is the circumradius of a regular pentagon of
diameter \(1\).  For any direction in the plane, one of the five
vertex directions makes angle at most $\pi/5$ with it.  Therefore
$R_5(z)$ contains a vertex whose distance from the origin is at least
\[
d_5(r_5)
=
\sqrt{
r_5^2
+
\left(R_5^{\rm vert}\right)^2
+
2r_5R_5^{\rm vert}\cos(\pi/5)
}.
\]

We use the cutoff
\[
b_5=0.194601.
\]

Again, $d_5(r)$ is strictly increasing for $r\ge0$, by the same
derivative formula with \(R=R_5^{\rm vert}\) and \(\delta=\pi/5\).  The
outward-rounded interval verification gives
\[
G(d_5(b_5))-0.834
\in
[1.46438884041707,\,1.46438884041708]\times 10^{-7},
\]
which is strictly positive.

Thus, if
\[
\|c_5\|\ge b_5,
\]
then the full convex hull contains the convex hull of $D$ and a point
$p\in R_5(z)$ with
\[
|p|\ge d_5(b_5),
\]
and hence
\[
A(z)\ge G(d_5(b_5))>0.834.
\]

This proves the second far-field exclusion.

\begin{lemma}\label{lem:far-r5}
If a normalized configuration satisfies
\[
A(z)<0.834,
\]
then
\[
\|c_5\|<b_5=0.194601.
\]
\end{lemma}

\subsection{The coarse bounded domain}

Combining Lemmas~\ref{lem:far-r3} and~\ref{lem:far-r5}, any
normalized configuration \(z\in\Omega_{\rm full}\) with area below
$0.834$ must satisfy
\[
\|c_3\|<0.190700,
\qquad
\|c_5\|<0.194601.
\]

Therefore every possible normalized counterexample to
\[
A(z)\ge0.834
\]
lies in the coarse bounded domain
\[
\Omega_{\rm coarse}^{\rm ball}
=
\left\{
(x_3,y_3,\psi_5,x_5,y_5):
\begin{array}{l}
\|(x_3,y_3)\|<0.190700,\\
0\le\psi_5\le \pi/5,\\
\|(x_5,y_5)\|<0.194601
\end{array}
\right\}.
\]

For interval subdivision it is more convenient to work with a
rectangular box containing this ball domain.  Introduce
\[
\Omega_{\rm coarse}
=
[-0.190700,0.190700]^2
\times
\left[0,\frac{\pi}{5}\right]
\times
[-0.194601,0.194601]^2.
\]

The analytic far-field reduction can then be stated as follows.

\begin{proposition}\label{prop:far-field}
For \(z\in\Omega_{\rm full}\), if
\[
A(z)<0.834,
\]
then
\[
z\in\Omega_{\rm coarse}.
\]
Equivalently, every normalized configuration
\(z\in\Omega_{\rm full}\setminus\Omega_{\rm coarse}\) is safe.
\end{proposition}

\begin{proof}
By Lemmas~\ref{lem:far-r3} and~\ref{lem:far-r5}, any normalized
configuration \(z\in\Omega_{\rm full}\) with $A(z)<0.834$ satisfies
\[
\|c_3\|<b_3,
\qquad
\|c_5\|<b_5.
\]
Hence
\[
|x_3|<b_3,\quad |y_3|<b_3,
\qquad
|x_5|<b_5,\quad |y_5|<b_5.
\]
Together with the angular normalization
$0\le\psi_5\le\pi/5$, this implies
\[
z\in
[-b_3,b_3]^2\times[0,\pi/5]\times[-b_5,b_5]^2
=
\Omega_{\rm coarse}.
\]
Thus any normalized configuration
\(z\in\Omega_{\rm full}\setminus\Omega_{\rm coarse}\)
satisfies \(A(z)\ge0.834\); in fact the far-field estimates give the
strict inequality in the relevant exterior regions.
\end{proof}

The reduction in this section removes the unbounded translation
variables.  The remaining coarse domain is still much larger than the
final core box used for the five-dimensional certificate.  The next
section removes the square collar regions inside \(\Omega_{\rm coarse}\)
in which
\[
\|c_3\|_\infty\ge 0.08
\quad\text{or}\quad
\|c_5\|_\infty\ge 0.08
\]
by one-body collar certificates, leaving only
\[
\Omega_{\rm core}
=
[-0.08,0.08]^2
\times
\left[0,\frac{\pi}{5}\right]
\times
[-0.08,0.08]^2.
\]

\section{Collar Certificates and Reduction to the Core Box}
\label{sec:collar}

Section~\ref{sec:analytic-reduction} reduced the unbounded
five-dimensional parameter space to the coarse bounded domain
\[
\Omega_{\rm coarse}
=
[-b_3,b_3]^2
\times
\left[0,\frac{\pi}{5}\right]
\times
[-b_5,b_5]^2,
\]
where
\[
b_3=0.190700,
\qquad
b_5=0.194601.
\]

Thus every possible normalized counterexample to
\[
A(z)\ge 0.834
\]
must lie in $\Omega_{\rm coarse}$.  The next reduction takes this coarse domain
further to a much smaller core box:
\[
\Omega_{\rm core}
=
[-0.08,0.08]^2
\times
\left[0,\frac{\pi}{5}\right]
\times
[-0.08,0.08]^2.
\]

The reduction from $\Omega_{\rm coarse}$ to $\Omega_{\rm core}$ is
performed by one-body collar certificates.  Intuitively, if either
$R_3$ or $R_5$ is outside the small central square but not already in
the far field, then the convex hull of the disk with that single
translated body already has area at least $0.834$.  Adding the other
body can only increase the convex hull area.  Hence only the core box
can contain potentially dangerous configurations.

\subsection{The collar regions}

Write
\[
c_3=(x_3,y_3),
\qquad
c_5=(x_5,y_5).
\]

Let
\[
a_0=0.08.
\]
This cutoff is a certificate-design parameter, not a special geometric
constant.  It is chosen strictly below the far-field cutoffs \(b_3\)
and \(b_5\), so that the annular square collars between \(a_0\) and
the far-field bounds can be certified by one-body interval lower bounds.
The choice is justified by the certificate itself: the collar
verification below proves that every normalized configuration in
\(\Omega_{\rm coarse}\setminus\Omega_{\rm core}\) is already safe.  A
different cutoff would be equally legitimate if the corresponding
collar certificate closed.

The core box is the set of normalized configurations satisfying
\[
\|c_3\|_\infty\le a_0,
\qquad
\|c_5\|_\infty\le a_0,
\qquad
0\le\psi_5\le\frac{\pi}{5}.
\]

In product form,
\[
\Omega_{\rm core}
=
[-a_0,a_0]^2
\times
\left[0,\frac{\pi}{5}\right]
\times
[-a_0,a_0]^2.
\]

If $z\in\Omega_{\rm coarse}\setminus\Omega_{\rm core}$, then either
\[
\|c_3\|_\infty\ge a_0
\qquad\text{or}\qquad
\|c_5\|_\infty\ge a_0.
\]
In particular, either $\|c_3\|\ge a_0$ or $\|c_5\|\ge a_0$.
The cases $\|c_3\|\ge b_3$ and $\|c_5\|\ge b_5$ have already been
removed by the far-field estimates of Section~\ref{sec:analytic-reduction}.
Thus the only remaining collar cases are the one-body radial domains
\[
C_3^{\rm collar}
=
[a_0,b_3]\times\left[0,\frac{\pi}{3}\right],
\qquad
C_5^{\rm collar}
=
[a_0,b_5]\times\left[0,\frac{\pi}{5}\right].
\]
The angular intervals are fundamental intervals for the regular
triangle and regular pentagon symmetries.  More explicitly, consider a
one-body hull
\[
\operatorname{conv}\bigl(D\cup(c+Q R_n^0)\bigr),
\]
where \(Q\) is an orthogonal placement of the reference regular
Reuleaux \(n\)-gon.  Since \(D\) is the centered disk, applying
\(Q^{-1}\) to the whole one-body configuration preserves the area and
reduces the body to \(R_n^0\), while changing only the direction of the
center.  The dihedral symmetry of \(R_n^0\) then identifies center
directions modulo reflection and rotation by \(2\pi/n\).  Hence it is
enough to certify the fundamental angle interval \([0,\pi/n]\), giving
\([0,\pi/3]\) for \(R_3\) and \([0,\pi/5]\) for \(R_5\).

The public collar certificates verify these two domains.  They prove
that a translated regular Reuleaux triangle or pentagon in the
corresponding radial collar already gives a disk-plus-body convex hull
of area at least $0.834$.

\subsection{One-body collar lower bounds}

The collar certificates use the following general principle.  Suppose
that a translated regular Reuleaux body $R_n(r,\theta)$ has center
with polar coordinates $(r,\theta)$, where $(r,\theta)$ lies in the
corresponding one-body collar domain $C_n^{\rm collar}$.  Then at
least one support point or short support arc of $R_n(r,\theta)$ is
sufficiently exposed relative to the disk $D$.  The convex hull of
$D$ with this exposed subset already gives a lower bound exceeding
$0.834$.

The certificate is organized as an interval statement.  For a
parameter box $B\subset C_n^{\rm collar}$, one constructs a finite set
of interval-enclosed witness points
\[
p_1(z),\ldots,p_m(z)
\in
D\cup R_n(z),
\qquad z\in B,
\]
where membership is certified by construction: each witness is either a
fixed disk sample or a fixed sample of the reference Reuleaux body,
translated by the interval center, and the fixed sample membership is
checked by outward-rounded inequalities.  The resulting polygon
\[
P_B(z)
=
\operatorname{conv}
\{p_1(z),\ldots,p_m(z)\}
\]
is contained in the corresponding disk-plus-body convex hull
\[
P_B(z)
\subseteq
\operatorname{conv}
\bigl(D\cup R_n(z)\bigr).
\]

Consequently,
\[
\operatorname{area}\operatorname{conv}\bigl(D\cup R_n(z)\bigr)
\ge
\operatorname{area}(P_B(z)).
\]

The local collar certificate verifies, by interval arithmetic, that
\[
\operatorname{area}(P_B(z))\ge0.834
\qquad
\text{for every }z\in B.
\]

The archived interval verifier checks both parts of this assertion: the recorded
collar boxes form an exact endpoint partition of the public collar
frontiers, and every leaf recomputes the local lower bound by the
approved one-body inner-polygon method.  The
union of the accepted collar boxes then proves the corresponding collar
exclusion.

\subsection{The \texorpdfstring{$R_3$}{R3}-collar certificate}

The $R_3$-collar certificate consists of a finite box decomposition of
\[
C_3^{\rm collar}
=
[0.08,0.190700]\times\left[0,\frac{\pi}{3}\right]
\]
such that every box $B_\alpha^{(3)}$ is certified safe:
\[
\inf_{(r,\theta)\in B_\alpha^{(3)}}
\operatorname{area}
\operatorname{conv}
\bigl(D\cup R_3(r,\theta)\bigr)
\ge0.834.
\]

Each local certificate uses only interval-enclosed geometric data from
the disk $D$ and the translated regular Reuleaux triangle
$R_3(r,\theta)$.  Thus the certificate proves
\[
\operatorname{area}
\operatorname{conv}
\bigl(D\cup R_3(r,\theta)\bigr)
\ge0.834
\qquad
\text{for all }(r,\theta)\in C_3^{\rm collar}.
\]

\subsection{The \texorpdfstring{$R_5$}{R5}-collar certificate}

The $R_5$-collar certificate is a finite box decomposition of
\[
C_5^{\rm collar}
=
[0.08,0.194601]\times\left[0,\frac{\pi}{5}\right]
\]
with a certified lower bound
\[
\inf_{(r,\theta)\in B_\beta^{(5)}}
\operatorname{area}
\operatorname{conv}
\bigl(D\cup R_5(r,\theta)\bigr)
\ge0.834
\]
on every box.  Hence
\[
\operatorname{area}
\operatorname{conv}
\bigl(D\cup R_5(r,\theta)\bigr)
\ge0.834
\qquad
\text{for all }(r,\theta)\in C_5^{\rm collar}.
\]

\subsection{Reduction to the core box}

Combining the far-field estimates with the two one-body collar
certificates gives the desired domain reduction.

\begin{proposition}\label{prop:collar}
Assume that the one-body collar certificates verify
\[
\operatorname{area}
\operatorname{conv}
\bigl(D\cup R_3(r,\theta)\bigr)
\ge0.834
\qquad
\text{on }C_3^{\rm collar}
\]
and
\[
\operatorname{area}
\operatorname{conv}
\bigl(D\cup R_5(r,\theta)\bigr)
\ge0.834
\qquad
\text{on }C_5^{\rm collar}.
\]
Then every normalized configuration $z\in\Omega_{\rm coarse}$ with
$A(z)<0.834$ must lie in $\Omega_{\rm core}$.
\end{proposition}

\begin{proof}
Let $z\in\Omega_{\rm coarse}\setminus\Omega_{\rm core}$.  Then either
$\|c_3\|_\infty\ge a_0$ or $\|c_5\|_\infty\ge a_0$.

If $\|c_3\|_\infty\ge a_0$, then $r_3=\|c_3\|\ge a_0$.  If
$r_3\ge b_3$, the far-field reduction gives $A(z)>0.834$.  Otherwise
$r_3\in[a_0,b_3]$.  By the one-body symmetry reduction above, applying
a dihedral symmetry of the regular Reuleaux triangle moves the center
direction into the fundamental interval without changing
\(\operatorname{area}\operatorname{conv}(D\cup R_3(z))\).  The
configuration is represented by some
\((r_3,\theta_3)\in C_3^{\rm collar}\), and the $R_3$ one-body collar
certificate gives
\[
\operatorname{area}
\operatorname{conv}
\bigl(D\cup R_3(z)\bigr)
\ge0.834.
\]
Since this hull is contained in
$\operatorname{conv}(D\cup R_3(z)\cup R_5(z))$, we get
$A(z)\ge0.834$.

If $\|c_5\|_\infty\ge a_0$, the same argument applies.  If
$r_5\ge b_5$, the far-field reduction gives $A(z)>0.834$.  Otherwise
\(r_5\in[a_0,b_5]\).  First rotate the one-body configuration so that
the placed Reuleaux pentagon has the reference orientation; this fixes
the disk and preserves the one-body hull area.  Then use the dihedral
symmetry of \(R_5^0\) to move the resulting center direction into
\([0,\pi/5]\).  The $R_5$ one-body collar certificate on
\(C_5^{\rm collar}\) gives
\[
\operatorname{area}
\operatorname{conv}
\bigl(D\cup R_5(z)\bigr)
\ge0.834,
\]
and hence again \(A(z)\ge0.834\).  Therefore no point in
\(\Omega_{\rm coarse}\setminus\Omega_{\rm core}\) can satisfy
\(A(z)<0.834\), so any such configuration must lie in \(\Omega_{\rm core}\).
\end{proof}

Together with Proposition~\ref{prop:far-field},
Proposition~\ref{prop:collar} gives the complete domain reduction.
The far-field estimates exclude
\(\Omega_{\rm full}\setminus\Omega_{\rm coarse}\), the collar
certificates exclude
\(\Omega_{\rm coarse}\setminus\Omega_{\rm core}\), and the only
remaining region is the compact core box.  Hence
\[
\inf_{z\in\Omega_{\rm full}} A(z)\ge0.834
\]
will follow once we prove
\[
\inf_{z\in\Omega_{\rm core}} A(z)\ge0.834.
\]

It remains to certify this core five-dimensional domain by interval
methods.

\section{The Core5D Interval Certificate}
\label{sec:core5d}

Sections~\ref{sec:analytic-reduction} and~\ref{sec:collar} reduced the
global finite-dimensional problem to the core box
\[
\Omega_{\rm core}
=
[-0.08,0.08]^2
\times
\left[0,\frac{\pi}{5}\right]
\times
[-0.08,0.08]^2.
\]

Thus, to prove the certified lower bound, it remains to verify
\[
A(z)\ge 0.834
\qquad
\text{for every }z\in\Omega_{\rm core},
\]
where
\[
A(z)
=
\operatorname{area}
\operatorname{conv}
\bigl(
D\cup R_3(z)\cup R_5(z)
\bigr).
\]

The Core5D interval certificate consists of a finite cover of the core
domain by boxes, together with local lower-bound proofs on those boxes.
The local geometric methods are proved in Section~\ref{sec:local-methods}.

\subsection{The core parameter domain}

The core parameter vector is
\[
z=(x_3,y_3,\psi_5,x_5,y_5),
\]
with
\[
(x_3,y_3)\in[-0.08,0.08]^2,
\]
\[
\psi_5\in\left[0,\frac{\pi}{5}\right],
\]
and
\[
(x_5,y_5)\in[-0.08,0.08]^2.
\]

Thus
\[
\Omega_{\rm core}
=
[-0.08,0.08]
\times[-0.08,0.08]
\times
\left[0,\frac{\pi}{5}\right]
\times
[-0.08,0.08]
\times[-0.08,0.08].
\]

For a parameter box
\[
B
=
I_{x_3}\times I_{y_3}\times I_{\psi_5}\times I_{x_5}\times I_{y_5}
\subseteq \Omega_{\rm core},
\]
the local certification problem is to prove
\[
\inf_{z\in B}A(z)\ge 0.834.
\]

If this is done for a finite family of boxes whose union covers
\(\Omega_{\rm core}\), then the desired global inequality follows.

\subsection{Box decomposition and certificate leaves}

The interval proof covers $\Omega_{\rm core}$ by finitely many boxes.
Conceptually, the certificate consists of a finite collection
of leaf boxes
\[
\mathcal B
=
\{B_\alpha\}_{\alpha\in I},
\]
such that
\[
\Omega_{\rm core}
\subseteq
\bigcup_{\alpha\in I}B_\alpha,
\]
and each box carries a local lower-bound certificate:
\[
\operatorname{LB}(B_\alpha)\ge0.834,
\]
where
\[
\operatorname{LB}(B_\alpha)
\le
\inf_{z\in B_\alpha}A(z)
\]
is a rigorously certified lower bound.

The subdivision may be adaptive.  Boxes in geometrically simple
regions can be large, while boxes near the possible minimum of $A$ may
need to be much smaller.  This adaptivity affects only the efficiency
of the proof, not its logical form.  The final certificate is finite
and can be checked independently.

\subsection{Certificate schema}
\label{sec:certificate-schema}

Each core leaf record contains the information needed for an independent
check of its local proof obligation.  At the mathematical level, a
leaf record is
of the form
\[
\mathsf{Leaf}(B,\mathsf{method},\mathsf{params}),
\]
where $B$ is a rational or outward-rounded interval box,
$\mathsf{method}$ is an approved local-method label, and
\(\mathsf{params}\) contains finite witness data such as the sampling
resolution and arithmetic precision.  The interval verifier treats this
record as a proof obligation.  It reconstructs all interval variables
from \(B\), checks that the method label belongs to the approved list,
and then recomputes a valid local certificate from \(B\) and
\(\mathsf{params}\).  A leaf is accepted only if one of the approved
methods produces a rigorous interval lower bound satisfying
\[
\operatorname{LB}(B)>0.834.
\]
If no approved method produces such a strict interval lower bound, the
leaf is rejected and the certificate fails.  Thus a successful certificate
has no uncertified leaf box.
The collar certificates of Section~\ref{sec:collar} are simpler: their
only approved method is
\[
\text{\texttt{PASS\_ONEBODY\_INNER\_POLYGON}}.
\]
In all cases, the
certificate depends on interval-certified local inequalities, not on a
floating-point search path or a stored numerical lower-bound claim.

The root record specifies the domain
\[
\Omega_{\rm core}
=
[-0.08,0.08]^2\times[0,\pi/5]\times[-0.08,0.08]^2.
\]
The interval verifier checks that the accepted leaf boxes form an exact
subdivision of the root box, with no gaps and only prescribed shared
faces.  This verifies both local safety and global coverage.

\subsection{Local lower-bound requirement}

For each leaf box $B$, the certificate must prove
\[
A(z)\ge0.834
\qquad
\text{for all }z\in B.
\]

Since
\[
A(z)
=
\operatorname{area}
\operatorname{conv}
\bigl(
D\cup R_3(z)\cup R_5(z)
\bigr),
\]
a lower bound for $A(z)$ may be obtained by constructing a polygonal
subset of the convex hull.

More precisely, suppose that for every $z\in B$ one can identify an
ordered polygon \(P_B(z)=[p_1(z),\ldots,p_m(z)]\) whose vertices satisfy
\[
p_1(z),\ldots,p_m(z)
\in
D\cup R_3(z)\cup R_5(z)
\]
and whose filled region is certified to lie inside the convex hull of
these vertices.  Then
\[
P_B(z)
\subseteq
\operatorname{conv}\{p_1(z),\ldots,p_m(z)\}
\subseteq
\operatorname{conv}
\bigl(
D\cup R_3(z)\cup R_5(z)
\bigr).
\]
In particular,
\[
A(z)\ge\operatorname{area}(P_B(z)).
\]
Thus it is sufficient to prove
\[
\operatorname{area}(P_B(z))\ge0.834
\qquad
\text{for all }z\in B.
\]

This formulation covers the different local methods used below.  In
the inner-polygon and template-convex methods, the witness polygon is
certified through convex or convex-template structure.  In the
mean-value simple-polygon method, the ordered polygon is certified to be
simple and positively oriented, and the filled polygonal region is then
used as the inner area witness.

The local methods in Section~\ref{sec:local-methods} differ in how the
points $p_i(z)$, the polygon $P_B(z)$, and its area lower bound are
certified.  In every case, the proof must be interval-valid on the
whole box $B$, not merely at its center.

\subsection{Frontier and leaf consistency}

A certificate is not complete merely because each leaf box is locally
safe.  It must also prove that the leaf boxes cover the entire core
domain without gaps.

The interval certification checks the following conditions:
\begin{enumerate}
\item \textbf{Domain containment.}
  Every leaf box satisfies $B_\alpha\subseteq\Omega_{\rm core}$.

\item \textbf{Coverage.}
  The union of all leaf boxes covers the core domain:
  $\Omega_{\rm core}\subseteq \bigcup_{\alpha\in I}B_\alpha$.

\item \textbf{Coverage and face consistency.}
  The accepted leaf boxes form a complete subdivision of the prescribed
  domain: adjacent boxes meet only along prescribed shared faces, and
  there are no positive-volume gaps or overlaps.

\item \textbf{Local safety.}
  For every leaf box $B_\alpha$, the verifier recomputes a rigorous lower
  bound
  $\operatorname{LB}(B_\alpha)\le\inf_{z\in B_\alpha}A(z)$
  and verifies $\operatorname{LB}(B_\alpha)\ge0.834$.
\end{enumerate}

When these checks pass, the certificate proves
\[
A(z)\ge0.834
\qquad
\text{for every }z\in\Omega_{\rm core}.
\]

\subsection{Interval arithmetic and outward rounding}

All local lower bounds are evaluated using interval arithmetic.  Each
coordinate interval in a box $B$ is treated as an interval variable.
Every operation that enters the lower-bound computation is performed
with outward rounding, so that the resulting interval encloses all
possible values over the whole box.  This is the standard validated
numerics paradigm for interval enclosures
~\cite{Moore1966,Neumaier1990,MooreKearfottCloud2009,Tucker2011}.

This is crucial because the certificate must rule out all
configurations in $B$, not just sampled points.  For example, if
$F(z)$ is a polygonal area expression used as a lower-bound witness,
interval arithmetic produces an interval $F(B)$ such that
\[
F(z)\in F(B)
\qquad
\text{for all }z\in B.
\]

If the lower endpoint of the interval \(F(B)\) is at least \(0.834\), then
$F(z)\ge0.834$ for every $z\in B$.  Since $A(z)\ge F(z)$, this proves
the safety of the box.

The details of the different choices of $F$ are given in
Section~\ref{sec:local-methods}.

\subsection{Certificate theorem for the core box}

The preceding discussion can be summarized as the following
certificate theorem.

\begin{proposition}\label{prop:core-cert}
Suppose that the core domain $\Omega_{\rm core}$ is covered by
finitely many parameter boxes
\[
\Omega_{\rm core}
\subseteq
\bigcup_{\alpha\in I}B_\alpha,
\]
and that for every leaf box $B_\alpha$ there is an interval-certified
lower bound
\[
\operatorname{LB}(B_\alpha)
\le
\inf_{z\in B_\alpha}A(z)
\]
satisfying
\[
\operatorname{LB}(B_\alpha)\ge0.834.
\]
Then
\[
A(z)\ge0.834
\qquad
\text{for every }z\in\Omega_{\rm core}.
\]
\end{proposition}

\begin{proof}
Let $z\in\Omega_{\rm core}$.  By the coverage condition, there exists
$\alpha\in I$ such that $z\in B_\alpha$.  By the local certificate on
$B_\alpha$,
\[
A(z)\ge \inf_{w\in B_\alpha}A(w)\ge \operatorname{LB}(B_\alpha)\ge0.834.
\]
Since $z\in\Omega_{\rm core}$ was arbitrary, the claim follows.
\end{proof}

Combining Proposition~\ref{prop:core-cert} with the reductions in
Sections~\ref{sec:analytic-reduction} and~\ref{sec:collar} reduces the
proof of $L_{3,5}^{\rm R}\ge0.834$ to the construction and verification
of a finite interval certificate over $\Omega_{\rm core}$.
Section~\ref{sec:local-methods} describes the local lower-bound
methods used to certify individual leaf boxes.

\section{Local Lower-Bound Methods}
\label{sec:local-methods}

On each Core5D leaf box, the proof reduces to a local lower bound for
the convex-hull area.

For a parameter box
\[
B
=
I_{x_3}\times I_{y_3}\times I_{\psi_5}\times I_{x_5}\times I_{y_5}
\subseteq\Omega_{\rm core},
\]
the local task is to prove
\[
A(z)
=
\operatorname{area}
\operatorname{conv}
\bigl(D\cup R_3(z)\cup R_5(z)\bigr)
\ge 0.834
\qquad
\text{for all }z\in B.
\]

All methods below follow the same principle:
\[
\text{construct a certified polygon inside the true convex hull}
\]
and then prove that the polygon has area at least $0.834$ throughout
the whole box.

The methods differ only in how the polygon is chosen and how its area
is bounded over the box.
For certification purposes, each accepted local method is an input-output certificate:
from a box \(B\) and recorded finite witness data it returns a lower bound
\(\operatorname{LB}(B)\) and proves
\[
\operatorname{LB}(B)\le \inf_{z\in B}A(z).
\]

\begin{center}
\begin{tabular}{lll}
\hline
method & certificate produced & area estimate\\
\hline
inner polygon &
labelled witness polygon &
interval shoelace\\
template convexity &
convex template &
perturbation bound\\
simple mean-value &
simple template &
mixed area bound\\
\hline
\end{tabular}
\end{center}

\subsection{Unified local certificate}

The common mathematical statement is the following.  Let \(\tau=0.834\).
A local polygon certificate on a box
\(B\) consists of an ordered polygon
\[
P_B(z)
=
[p_1(z),\ldots,p_m(z)]
\]
defined for every \(z\in B\).  When areas or set inclusions are
discussed, \(P_B(z)\) denotes the filled polygonal region bounded by
this ordered polygon.

\begin{proposition}[Local polygon certificate]\label{prop:local-sound}
Suppose that, for every \(z\in B\), the polygon \(P_B(z)\) satisfies:
\begin{enumerate}
\item each vertex \(p_i(z)\) belongs to \(D\cup R_3(z)\cup R_5(z)\);
\item \(P_B(z)\) is convex and positively oriented, or is simple and
  positively oriented;
\item \(\operatorname{area}(P_B(z))\ge\tau\).
\end{enumerate}
Then
\[
A(z)\ge\tau
\qquad
\text{for every }z\in B.
\]
\end{proposition}

\begin{proof}
By the first condition,
\[
\operatorname{conv}\{p_1(z),\ldots,p_m(z)\}
\subseteq
\operatorname{conv}\bigl(D\cup R_3(z)\cup R_5(z)\bigr).
\]
By the second condition, the filled polygon lies inside the convex hull
of its vertices.  Hence
\[
P_B(z)
\subseteq
\operatorname{conv}\{p_1(z),\ldots,p_m(z)\}
\subseteq
\operatorname{conv}
\bigl(D\cup R_3(z)\cup R_5(z)\bigr).
\]
Taking areas and using the third condition gives
\[
A(z)
=
\operatorname{area}\operatorname{conv}\bigl(D\cup R_3(z)\cup R_5(z)\bigr)
\ge
\operatorname{area}(P_B(z))
\ge \tau .
\]
\end{proof}

\begin{lemma}[Turn criterion for convexity]\label{lem:turn-convexity}
Let \(P=[p_1,\ldots,p_m]\) be a closed polygon, and set
\(e_i=p_{i+1}-p_i\), cyclically.  Suppose that no edge vanishes, that
\[
\det(e_i,e_{i+1})>0
\qquad (i=1,\ldots,m),
\]
and that the edge-direction map has rotation index one.  Then \(P\) is
a strictly convex positively oriented polygon.
\end{lemma}

\begin{proof}
Choose continuous arguments \(\alpha_i\) for the edge directions along
one turn of the polygon.  The strict inequalities
\(\det(e_i,e_{i+1})>0\) mean that each exterior turn
\(\alpha_{i+1}-\alpha_i\), taken in \((0,\pi)\), is strictly positive.
Rotation index one says that the total turn is \(2\pi\).  Hence the edge
directions wind once around \(S^1\) with strictly increasing arguments.
The standard turning-angle criterion for polygons then gives a simple
positively oriented boundary whose supporting direction changes
monotonically once around the circle; the polygon is strictly convex.
\end{proof}

The remaining task in a leaf box is to verify the three
hypotheses of Proposition~\ref{prop:local-sound}.  The third hypothesis
is usually checked through the signed shoelace area.

For an ordered polygon with vertices
\[
p_i(z)=(X_i(z),Y_i(z)),
\qquad i=1,\ldots,m,
\]
the signed area is
\[
S_B(z)
=
\frac12
\sum_{i=1}^m
\left(
X_i(z)Y_{i+1}(z)-Y_i(z)X_{i+1}(z)
\right),
\]
with the convention $p_{m+1}=p_1$.  If the polygon is certified to be
positively oriented and simple, then
$\operatorname{area}(P_B(z))=S_B(z)$.

If it is certified to be convex and positively oriented, this equality
is immediate.  If it is only certified to be simple and positively
oriented, the shoelace formula still gives its area.

Therefore the local methods below all have the same mathematical
shape: they certify the vertices, certify convexity or simplicity with
positive orientation, and prove
\[
S_B(z)\ge\tau
\qquad
(z\in B),
\]
together with the required geometric condition.

For certification purposes, the three methods share the same logical
interface:
\[
(B,\mathsf{params})
\longmapsto
(\mathsf{pass},\operatorname{LB}(B)).
\]
The interval verifier reconstructs the witness polygon or template from
\(B\) and the finite witness data.  A passed local certificate means
precisely that the hypotheses of Proposition~\ref{prop:local-sound} have
been certified with \(\operatorname{LB}(B)\ge\tau\).

\subsection{Direct inner-polygon method}

The first method constructs a polygon directly from certified support
points of the three bodies.

\smallskip
\noindent\textbf{Method~1: Direct inner-polygon certificate.}
For a leaf box $B$, choose a finite ordered list of labels
$\ell_1,\ldots,\ell_m$.  Each label $\ell_i$ specifies a point on one
of the bodies $D$, $R_3(z)$, or $R_5(z)$.  For example, a label may
specify:
\begin{itemize}
\item a fixed point on the disk $D$;
\item a vertex of $R_3(z)$;
\item a point on a circular arc of $R_3(z)$;
\item a vertex of $R_5(z)$;
\item a point on a circular arc of $R_5(z)$.
\end{itemize}

Each label determines a point function \(p_i(z)\) on the corresponding
body for every \(z\in B\).  Thus the membership condition
\[
p_i(z)\in D\cup R_3(z)\cup R_5(z)
\]
is certified by the label and by the interval construction of the point.
Outward-rounded interval arithmetic is then used to enclose its
coordinates:
\[
X_i(z)\in X_i(B),
\qquad
Y_i(z)\in Y_i(B)
\qquad
(z\in B).
\]

Using interval arithmetic, compute an interval enclosure \(S_B(B)\) for
the signed shoelace area \(S_B(z)\).  If the lower endpoint satisfies
\(\underline S_B\ge\tau\), and if the polygon is certified to be convex
and positively oriented or simple and positively oriented throughout
\(B\), then Proposition~\ref{prop:local-sound} gives
\(A(z)\ge\tau\) on the whole box.

This method is robust and simple.  It is especially effective on boxes
where the active supporting structure of the convex hull is stable and
where direct interval evaluation of the polygonal area is not too
wide.

We denote this local pass by
\[
\texttt{PASS\_CORE5D\_INNER\_POLYGON}.
\]

\subsection{Template convexity method}

The second method improves the direct inner-polygon method by fixing a
combinatorial support template.

The main difficulty in direct interval evaluation is that the chosen
support polygon may vary as \(z\) varies in a box.  On a sufficiently
small box, however, a fixed ordered template can often be certified over
the entire box.

\smallskip
\noindent\textbf{Method~2: Template convexity certificate.}
Let $I=(\ell_1,\ldots,\ell_m)$ be an ordered support template.  For
each $z\in B$, the labels define points
$p_1(z),\ldots,p_m(z)\in D\cup R_3(z)\cup R_5(z)$.

The template certificate verifies the following two properties over the
whole box $B$.  First, the polygon
$P_I(z)=[p_1(z),\ldots,p_m(z)]$ is positively oriented and convex.  This
is certified by a perturbation argument.  The template is chosen as a
cyclic subsequence of the convex hull of the center configuration, so
the center polygon is a strictly convex positively oriented polygon with
rotation index one.  Each vertex is assigned a rigorous displacement
radius \(\varepsilon_i\), valid throughout \(B\):
\[
\|p_i(z)-p_i(c_B)\|\le \varepsilon_i
\qquad (z\in B).
\]
For every consecutive triple, the interval verifier checks a strict lower turn
bound over the full product of these displacement balls.  Equivalently,
for every choice of points \(q_i\) satisfying
\(\|q_i-p_i(c_B)\|\le\varepsilon_i\), the consecutive turn
\[
\det
\left(
q_{i+1}-q_i,
q_{i+2}-q_{i+1}
\right)
\]
has a positive lower bound, cyclically in \(i\).  This is implemented as
a center-turn margin minus a certified turn-perturbation bound.

These strict turn bounds also exclude collapsed consecutive edges: if
an edge in the displacement balls vanished, an adjacent turn would be
zero.  For a fixed \(z\in B\), the straight-line homotopy
\[
q_i(t)=p_i(c_B)+t\bigl(p_i(z)-p_i(c_B)\bigr),
\qquad 0\le t\le1,
\]
stays inside the same displacement balls.  Hence all consecutive turns
remain strictly positive along the homotopy, no edge vanishes, and the
edge-direction rotation index stays equal to the center value, namely
one.  By Lemma~\ref{lem:turn-convexity}, the turn-margin check certifies
global convexity, not merely local left turns.

Second, the signed area of this template polygon is bounded below.  The
center signed area is evaluated with outward rounding; write
\(\underline S_I(c_B)\) for the lower endpoint of this center-area
enclosure.  The basic bound is this center lower endpoint minus a
certified perturbation error coming from the same displacement radii:
\[
S_I(z)
\ge
\underline S_I(c_B)-E_I(B)
\ge0.834
\qquad
(z\in B).
\]
Here \(E_I(B)\) is a certified perturbation bound for the variation of
the template area over \(B\).
The same template may also be certified by the mean-value area bound
described in the next subsection.

Since every \(p_i(z)\) lies on one of the three test bodies, the three
hypotheses of Proposition~\ref{prop:local-sound} hold.  Thus
\(A(z)\ge\tau\) on \(B\).

This method is more structured than the direct inner-polygon method.
It is useful when the active hull template is stable across $B$, and
the center-template perturbation bounds, optionally sharpened by the
mean-value estimate, can prove convexity and area lower bounds without
excessive overestimation.

We denote this pass by
\[
\texttt{PASS\_CORE5D\_TEMPLATE\_CONVEX\_BOUND}.
\]

\subsection{Simple-template mean-value method}

The third method is the simple-template version of the preceding
template certificate.  It is used when convexity is not certified, but
the ordered template is still certified to be simple and positively
oriented on the whole box.  Convexity is not required here: simplicity
and positive orientation are enough to identify the signed shoelace area
with the filled area in Proposition~\ref{prop:local-sound}.

Let $c_B$ be the center of a parameter box $B$, and let
$r_B=(r_1,\ldots,r_5)$ be its half-width vector.  Suppose a support
template $I$ is chosen and gives a polygon area function $S_I(z)$.

Assume first that the polygon is certified to be simple and positively
oriented throughout $B$.  Convexity is sufficient but not necessary
for this method; simplicity plus orientation is enough for the signed
shoelace area to equal the polygon area.

The mean-value method estimates $S_I(z)$ by comparing it with the
center value along the line segment from \(c_B\) to \(z\).  For
\(z\in B\),
\[
S_I(z)-S_I(c_B)
=
\int_0^1
\nabla S_I\bigl(c_B+t(z-c_B)\bigr)\cdot (z-c_B)\,dt .
\]
Since \(B\) is convex, the segment \(c_B+t(z-c_B)\) stays in \(B\).
Therefore
\[
S_I(z)
\ge
\underline S_I(c_B)
-
\sum_{k=1}^5
G_k(B)r_k,
\]
where
\[
G_k(B)
\ge
\sup_{z\in B}
\left|
\frac{\partial S_I}{\partial z_k}(z)
\right|
\]
is an interval-certified derivative bound.  It is obtained by
differentiating the shoelace expression symbolically with respect to the
five parameters and evaluating the resulting expressions by
outward-rounded interval arithmetic on \(B\).  The center area
\(\underline S_I(c_B)\) is the lower endpoint of an outward-rounded
evaluation of the center shoelace area.  The derivative expressions
involve only the derivatives of translated or rotated support points.
For a point of the form
\[
p(z)=T_{(x_5,y_5)}Q_{\psi_5}p_0,
\]
for example, one uses
\[
\frac{\partial p}{\partial x_5}=(1,0),
\qquad
\frac{\partial p}{\partial y_5}=(0,1),
\qquad
\frac{\partial p}{\partial \psi_5}=Q_{\psi_5}Jp_0,
\]
where $J$ is rotation by $\pi/2$.  Arc points are handled in the same
way: the fixed or interval-enclosed arc parameter is included among the
interval inputs when the point coordinate derivatives are evaluated.
Thus \(G_k(B)\) is a rigorous supremum bound over the whole box and the
relevant template-parameter intervals, not a sampled gradient.

Let
\[
L_{\rm pert}(B)=\underline S_I(c_B)-E_I(B)
\]
be the perturbation lower bound from the previous subsection, and let
\[
L_{\rm mv}(B)
=
\underline S_I(c_B)-\sum_{k=1}^5G_k(B)r_k
\]
be the mean-value lower bound.  Both are certified lower bounds for
\(\inf_{z\in B}S_I(z)\).  If either certified bound is at least
\(\tau\), then
\[
S_I(z)\ge\tau
\qquad
(z\in B).
\]

Together with the certified simplicity and positive orientation of the
template polygon, Proposition~\ref{prop:local-sound} proves
\(A(z)\ge\tau\) for every \(z\in B\).

The advantage of this method is that the center area is evaluated
sharply with outward rounding, while only the variation over the box is
bounded by interval derivatives.  This is often much less conservative
than direct interval evaluation of the whole shoelace expression.

We denote this pass by
\[
\texttt{PASS\_CORE5D\_TEMPLATE\_MEAN\_VALUE\_SIMPLE\_BOUND}.
\]

\subsection{Certification of simplicity and orientation}

The simple-template method requires a verified simple polygon on the
entire box.  The template convexity method requires an even stronger
convexity certificate.  We record the relevant checks.

Simplicity is certified by excluding edge crossings.  For each
nonadjacent pair of edges
\[
[p_i(z),p_{i+1}(z)]
\quad\text{and}\quad
[p_j(z),p_{j+1}(z)],
\]
one checks, using interval arithmetic, that their oriented separation
signs cannot simultaneously allow an intersection.  Equivalently, for
each nonadjacent pair, at least one standard segment-intersection
predicate is interval-separated away from zero.

For a polygon already certified to be simple, positive orientation is
certified by showing that its signed shoelace area is positive:
\[
S_P(z)>0
\qquad (z\in B).
\]

Convexity is certified by the stricter turn-margin framework used in
the template-convex method.  In particular, the consecutive turn
conditions
\[
\det
\left(
p_{i+1}(z)-p_i(z),
p_{i+2}(z)-p_{i+1}(z)
\right)
>0
\]
are used with strict interval margins over the allowed displacement
balls.  These margins exclude edge degeneracy, and the straight-line
homotopy from the center template keeps rotation index one.  The
cross-product conditions together with edge nondegeneracy and this
rotation-index-one homotopy certify global convexity by
Lemma~\ref{lem:turn-convexity}.

Once simplicity and positive orientation are certified, the signed
shoelace area equals the geometric area enclosed by the filled polygon.
For a convex positively oriented polygon the same conclusion is
immediate.  Finally, the certified vertex membership gives
\[
\operatorname{conv}\{p_1(z),\ldots,p_m(z)\}
\subseteq
\operatorname{conv}\bigl(D\cup R_3(z)\cup R_5(z)\bigr).
\]
The filled polygonal region is contained in the convex hull of its
vertices, both in the convex case and in the certified-simple case.
Thus the local geometric checks give the convex-hull inclusion required
by Proposition~\ref{prop:local-sound}.

\subsection{Why multiple local methods are used}

No single local method is uniformly optimal on the entire core box.
The direct inner-polygon method is simple and robust, but can be too
conservative on boxes where interval overestimation is large.  The
template convexity method is stronger when the active convex-hull
combinatorics are stable.  The simple-template mean-value method is
sharper near delicate regions where the area margin is small but varies
smoothly.

The certificate therefore allows each leaf box to be certified by any
one of the approved local methods:
\[
\texttt{PASS\_CORE5D\_INNER\_POLYGON},
\]
\[
\texttt{PASS\_CORE5D\_TEMPLATE\_CONVEX\_BOUND},
\]
or
\[
\texttt{PASS\_CORE5D\_TEMPLATE\_MEAN\_VALUE\_SIMPLE\_BOUND}.
\]

This flexibility affects the efficiency and size of the certificate,
but not its logical structure.  Every accepted leaf box must ultimately
prove the same mathematical statement:
\[
\inf_{z\in B}A(z)\ge0.834.
\]

Each approved local method is simply a different way of verifying
the hypotheses of Proposition~\ref{prop:local-sound}.  It remains to
check that the leaves cover the core domain and that every leaf box
carries a valid interval certificate.

\section{Interval Certification and the Certified \texorpdfstring{$0.834$}{0.834} Theorem}
\label{sec:interval-certification}

The preceding reductions leave a finite collection of verifiable
interval statements.  Their assembly into a certificate gives the final
lower-bound theorem.

The certification step is logically important.  The search procedure may
generate boxes, select templates, and record which local method succeeds
on each leaf.  These recorded data are accepted only after deterministic
interval checking of every assertion needed for the theorem:
\[
A(z)\ge0.834
\qquad
\text{for all }z\in\Omega_{\rm full}.
\]
The interval verifier returns \texttt{PASS} for the archived certificate;
the certificate is then a finite proof of the lower bound.

\subsection{Certificate components}

\subsubsection*{Fixed constants}

The certificate records all numerical constants used in the proof,
including
\[
b_3=0.190700,
\qquad
b_5=0.194601,
\qquad
a_0=0.08,
\]
and the target value
\[
\tau=0.834.
\]
It also records the angular domain $0\le\psi_5\le\pi/5$, and the core
box $\Omega_{\rm core}=[-0.08,0.08]^2\times[0,\pi/5]\times[-0.08,0.08]^2$.
All constants are interpreted as intervals with outward rounding in
the interval certificate.

\subsubsection*{Analytic reduction records}

The certificate records the analytic far-field exclusions from
Section~\ref{sec:analytic-reduction}:
\[
\|c_3\|\ge b_3 \;\Longrightarrow\; A(z)>\tau,
\qquad
\|c_5\|\ge b_5 \;\Longrightarrow\; A(z)>\tau.
\]
These checks depend only on the disk-plus-point function
\[
G(\rho)
=
\frac{\pi}{4}
+
\frac12\sqrt{\rho^2-\frac14}
-
\frac14
\arccos\left(\frac{1}{2\rho}\right),
\]
and on outward-rounded verification of
$G(d_3(b_3))>\tau$ and $G(d_5(b_5))>\tau$.

\subsubsection*{Collar certificates}

The certificate records the two one-body collar verifications on
$C_3^{\rm collar}$ and $C_5^{\rm collar}$.  They prove
\[
\operatorname{area}\operatorname{conv}\bigl(D\cup R_3(r,\theta)\bigr)
\ge\tau
\quad
\text{on }C_3^{\rm collar},
\]
and
\[
\operatorname{area}\operatorname{conv}\bigl(D\cup R_5(r,\theta)\bigr)
\ge\tau
\quad
\text{on }C_5^{\rm collar}.
\]
These collar certificates, together with the far-field exclusions,
prove that every normalized configuration in the coarse bounded domain
but outside the core box is already safe.
The collar leaves are certified by the one-body inner-polygon method.
In the archived implementation this method has label
\[
\texttt{PASS\_ONEBODY\_INNER\_POLYGON}.
\]

\subsubsection*{Core5D leaf certificates}

The main finite certificate consists of a collection of leaf boxes
$\mathcal B=\{B_\alpha\}_{\alpha\in I}$ covering the core box.  Each
leaf box records the box and finite witness data needed to certify a
local lower bound $\inf_{z\in B_\alpha}A(z)\ge\tau.$

Each leaf is certified by one of three approved mathematical local
methods: the Core5D inner-polygon method, the template-convexity method,
or the simple-template mean-value method.  The corresponding archived
implementation labels are recorded in the certificate documentation.

The stored method label must belong to the approved list.  Acceptance
depends on reconstructing a successful interval lower-bound certificate
on the recorded box.  A core leaf is accepted exactly when the interval
verifier produces a rigorous proof by one of the approved final methods
with recomputed lower bound strictly greater than \(\tau\).  If no
approved method succeeds on a leaf, the certificate fails.

\subsubsection*{Certificate status}

The archived certificate has status
\[
\texttt{PASS}.
\]
The interval verifier deterministically checks all constants, analytic
reductions, collar certificates, exact collar/core coverage, and local
interval lower bounds from the archived data.  Its role is to
instantiate the finite hypotheses of the global soundness theorem below.

\begin{center}
\begin{tabular}{ll}
\hline
Certificate status & \texttt{PASS}\\
Arithmetic & outward-rounded interval arithmetic\\
Target value & \(0.834\)\\
Certificate records & \(6{,}747{,}165\)\\
Errors and warnings & \(0\)\\
Checks verified & constants, reductions, exact coverage, local bounds\\
\hline
\end{tabular}
\end{center}

The certificate-record count includes the archived local certificates,
coverage summaries, and metadata needed for the finite verification.
The following margins are the audit quantities that bridge the displayed
mathematical constants, the recorded rational certificate inputs, and the
target value \(0.834\).
\begin{center}
\begin{tabular}{ll}
\hline
smallest strict certificate margin above \(0.834\) & \(>2.6\cdot10^{-11}\)\\
largest angular endpoint recording gap & \(<2\cdot10^{-16}\)\\
induced Hausdorff-area perturbation & \(<10^{-14}\)\\
errors and warnings in final audit & \(0\)\\
\hline
\end{tabular}
\end{center}

The certified parts have the following minimum interval lower-bound
enclosures.  The displayed lower bounds are midpoint-radius interval
enclosures; in each row, the lower endpoint of the displayed interval is
strictly larger than \(0.834\).
\[
\begin{array}{c|r|r|c}
\text{part} & \text{frontiers} & \text{leaves} &
\text{minimum interval lower-bound enclosure}\\
\hline
R_3\text{-collar} & 256 & 2{,}374 &
[0.83400005962221344 \pm 2.69\cdot10^{-30}]\\
R_5\text{-collar} & 256 & 1{,}029 &
[0.83400044554297159 \pm 4.31\cdot10^{-30}]\\
\text{Core5D} & 256 & 6{,}742{,}989 &
[0.83400000002676268 \pm 6.97\cdot10^{-31}]
\end{array}
\]

The public certificate package consists of the compressed certificate
shards, the standalone interval verifier, and the final audit
report, and is available in the public archive~\cite{CertificateArchive}.
The certificate shard set is identified by the SHA256 hash
\begin{center}
\small\ttfamily
c5fdc18d4e38f65d54e4e730e36ef1e41166ecc3a991494e653825f667eec78a
\end{center}
This checksum identifies the archived certificate shard set; the
standalone interval verifier and final audit report are included
in the same public archive.

\subsection{Reproducibility of the interval certificate}

The public archive is organized so that the interval certificate can be
verified without the certificate-generation code.  After unpacking the
archive, the certificate directory contains
\begin{center}
\begin{tabular}{p{0.32\textwidth}p{0.52\textwidth}}
\texttt{final\_audit.json} & final audit from a full verification\\
\texttt{shards\_hash.json} & SHA256 record for \texttt{shards/}\\
\texttt{shards/} & compressed certificate records\\
\texttt{verify\_certificate.py} & standalone verifier entry point\\
\texttt{verifier/} & verifier implementation
\end{tabular}
\end{center}
The directory \texttt{shards/} contains the compressed certificate
records.  The standalone verifier is contained in
\texttt{verify\_certificate.py} and \texttt{verifier/}; it only checks
the archived certificate and does not generate or repair certificates.

The verifier is written in Python and uses \texttt{python-flint} for
outward-rounded interval arithmetic.  The archived certificate was tested
with Python~3.9 and \texttt{python-flint}.  It was verified by running
the standalone verifier on the certificate directory.
A representative command is
\begin{verbatim}
python3 verify_certificate.py --audit-file final_audit_rerun.json
\end{verbatim}
or, with an explicit worker count,
\begin{verbatim}
python3 verify_certificate.py \
  --workers 4 \
  --audit-file final_audit_rerun.json
\end{verbatim}
By default the verifier uses \(\max(1,\texttt{cpu\_count}-1)\) workers.
The full verification may take several hours to tens of hours on an
ordinary multi-core desktop or laptop, depending on CPU speed, memory
bandwidth, and the number of workers.

A successful run writes a JSON audit whose key status fields are
\[
\texttt{coverage\_status}=\texttt{PASS},
\qquad
\texttt{status}=\texttt{PASS},
\qquad
\texttt{proof\_complete}=\texttt{true}.
\]
The archived final audit has these status fields, no errors, no
warnings, and \(6{,}747{,}165\) certificate rows.  The verifier checks
the two metadata proof records \(\texttt{constants}\) and
\(\texttt{reduction\_constants}\), then checks the three public geometric
parts \(\texttt{r3\_collar}\), \(\texttt{r5\_collar}\), and
\(\texttt{core5d}\).

The hash in \(\texttt{shards\_hash.json}\) applies to \texttt{shards/}
only.  It is computed from the sorted relative path, a null separator,
the file bytes, and another null separator for every file under
\texttt{shards/}.  In the archived package this shard set contains
\(773\) files and \(787{,}849{,}048\) bytes, with SHA256 hash
\begin{center}
\small\ttfamily
c5fdc18d4e38f65d54e4e730e36ef1e41166ecc3a991494e653825f667eec78a
\end{center}

\subsection{Certification checks}

The interval verifier performs the following checks.

\subsubsection*{Constant and domain checks}

The verifier first reconstructs the constants and domains used in the
proof.  It verifies that the core box is exactly the domain stated in
the paper:
$\Omega_{\rm core}=[-0.08,0.08]^2\times[0,\pi/5]\times[-0.08,0.08]^2$.
It then verifies that every recorded leaf box lies within the
appropriate domain.  The recorded constants are decimal rational inputs;
their relation to the ideal constants displayed in the text is handled
by the explicit perturbation bridge described below.

\subsubsection*{Far-field checks}

The verifier checks the analytic inequalities used in
Section~\ref{sec:analytic-reduction}, with outward rounding:
$G(d_3(b_3))>\tau$ and $G(d_5(b_5))>\tau$.
These checks certify that any possible violation of $A(z)\ge\tau$ must
lie in $\Omega_{\rm coarse}$.

\subsubsection*{Collar checks}

The verifier checks all collar leaves.  The result is the pair of
one-body certified statements on $C_3^{\rm collar}$ and
$C_5^{\rm collar}$.  Together with the far-field checks, this proves
that every possible violation must lie in $\Omega_{\rm core}$.

\subsubsection*{Coverage checks}

The verifier checks that the recorded leaves cover each public
geometric domain: $C_3^{\rm collar}$, $C_5^{\rm collar}$, and
$\Omega_{\rm core}$.  In each case the accepted leaf boxes form a
complete subdivision of the corresponding prescribed domain, with no
gaps and only the prescribed shared faces; equivalently, with no
positive-volume overlaps.

The constants used by the certificate are recorded as decimal rational inputs.
These recorded inputs are the exact data consumed by the verifier.
Several of them represent ideal mathematical constants displayed in the
paper, such as angular endpoints and vertex radii.  The bridge from the
recorded rational inputs to the ideal constants is verified by
outward-rounded perturbation bounds.  In the largest such case, the
angular endpoint gap is less than \(2\cdot10^{-16}\); the induced
Hausdorff-area perturbation is less than \(10^{-14}\), while the minimum
strict certificate margin is larger than \(2.6\cdot10^{-11}\).  The rounding
perturbations of the vertex-radius constants are smaller still.
Therefore the recorded rational inputs, together with the verified
perturbation bridge, prove the lower bound for the mathematical domains
and ideal constants stated in the paper.

\subsubsection*{Local lower-bound certification}

For every collar or core leaf box $B_\alpha$, the verifier recomputes a
rigorous lower bound from the recorded box and finite witness data.  The
stored method label identifies an approved class of local certificate.
Collar leaves use the one-body method recorded in the leaf; core leaves
are accepted when one of the approved final local methods produces a
rigorous interval proof with recomputed lower bound at least \(\tau\) on
the whole box.

The verification does not rely on auxiliary numerical claims or on a
lower bound stored with the certificate record.  It recomputes the
interval proof directly.

\subsubsection*{Final certificate status}

The certificate succeeds only after all constant, domain, reduction,
coverage, perturbation, and local lower-bound checks pass.  For the
archived certificate used in this paper, these checks all pass; hence the
certificate proves \(A(z)\ge\tau\) for all \(z\in\Omega_{\rm full}\).

\begin{proposition}[Certified interval instance]\label{prop:certified-instance}
The archived interval certificate verifies all finite hypotheses used in
the proof below: the stated constants and domains, the far-field
exclusions, the collar certificates, the collar/core coverage checks, and
the local lower-bound checks at target \(\tau=0.834\).  The verifier
returns \(\texttt{PASS}\) with no errors or warnings.
\end{proposition}

\begin{proof}
This is the deterministic interval verification of the archived certificate
package.  The verifier reconstructs the interval computations from the
recorded boxes and finite witness data, uses outward rounding throughout,
and accepts only after all checks listed above pass with no errors or
warnings.
\end{proof}

\subsection{Global soundness theorem}

The global proof follows by assembling the preceding reductions.

\begin{theorem}\label{thm:L35-lower}
The finite Reuleaux test value satisfies
\[
L_{3,5}^{\rm R}\ge0.834.
\]
\end{theorem}

\begin{proof}
By Proposition~\ref{prop:certified-instance}, the finite interval
certificate verifies the far-field exclusions, collar certificates, core
coverage, and all local lower-bound checks at target \(0.834\).
Let $z\in\Omega_{\rm full}$ be arbitrary.

If $z\notin\Omega_{\rm coarse}$, then the far-field reductions of
Section~\ref{sec:analytic-reduction} imply $A(z)>0.834$.

If $z\in\Omega_{\rm coarse}\setminus\Omega_{\rm core}$, then the
collar reductions of Section~\ref{sec:collar} imply $A(z)\ge0.834$.
Indeed, the relevant one-body convex hull is contained in the full
three-body convex hull, so adding the remaining body cannot decrease the
area.

It remains to consider the case $z\in\Omega_{\rm core}$.  By the core
coverage check, there exists a leaf box $B_\alpha$ such that
$z\in B_\alpha$.  By the local interval certificate on this leaf,
$A(z)\ge\inf_{w\in B_\alpha}A(w)\ge0.834$.

In all cases, $A(z)\ge0.834$.  Since $z\in\Omega_{\rm full}$ was
arbitrary,
\[
\inf_{z\in\Omega_{\rm full}}A(z)\ge0.834.
\]
By the five-parameter normalization proved in
Section~\ref{sec:normalization},
$L_{3,5}^{\rm R}=\inf_{z\in\Omega_{\rm full}}A(z)$, so
$L_{3,5}^{\rm R}\ge0.834$.
\end{proof}

\subsection{Certified lower bound for \texorpdfstring{$\aLeb$}{aLeb}}

The finite test value $L_{3,5}^{\rm R}$ is a lower bound for the Lebesgue
constant in the sense that
\[
L_{3,5}^{\rm R}\le \aLeb.
\]
This was proved in Section~\ref{sec:low-order}, because
$\{D,R_3,R_5\}\subset\mathcal W_1$.

Combining this with Theorem~\ref{thm:L35-lower} gives the certified
lower bound.

\begin{theorem}\label{thm:certified}
The Lebesgue universal covering constant satisfies
\[
\aLeb\ge0.834.
\]
\end{theorem}

\begin{proof}
By Theorem~\ref{thm:L35-lower}, $L_{3,5}^{\rm R}\ge0.834$.  By the finite-test
lower-bound property, $L_{3,5}^{\rm R}\le \aLeb$.  Hence
$\aLeb\ge L_{3,5}^{\rm R}\ge0.834$.
\end{proof}


\end{document}